\documentclass[3p,times]{elsarticle}

\usepackage{graphicx}
\usepackage{psfrag}
\usepackage[dvipsnames]{xcolor}

\usepackage{physics}
\usepackage{amsmath}
\usepackage{amsthm}
\usepackage{caption}
\usepackage{amsfonts}
\usepackage{bm}
\usepackage{mathbbol}
\usepackage{siunitx}
\usepackage{subcaption}
\newtheorem{theorem}{Theorem}
\newtheorem{remark}{Remark}
\usepackage{mathtools}
\usepackage{booktabs}

\usepackage{hyperref} 
\hypersetup{
	colorlinks=true,             
	linkcolor=blue,
	citecolor=NavyBlue,
	urlcolor=blue }

\newcommand{\vv}{\bm{v}}
\newcommand{\xx}{\bm{x}}

\newcommand{\ff}{\bm{f}}
\newcommand{\oxx}{\bm{\overline{x}}}
\newcommand{\comment}[1]{}

\newcommand{\sgrad}[2]{  \langle \nabla #1  \rangle_{#2}}
\newcommand{\sdiv}[2]{  \langle \divg #1  \rangle_{#2}}
\newcommand{\adgrad}[2]{ \langle \nabla^* #1 \rangle_{#2}}
\newcommand{\addiv}[2]{  \langle \divg^* #1  \rangle_{#2}}

\DeclareMathOperator{\divg}{div}
\DeclareMathOperator{\rot}{curl}

\newcommand{\dt}{\Delta t} 

\begin{document}
	
	\begin{frontmatter}
		
		\journal{Computers \& FLuids}
		
		
		
		\title{Semi-Implicit Lagrangian Voronoi Approximation for Compressible Viscous Fluid Flows}
		
		
		\author[unitn]{Ondřej Kincl$^*$}
		\ead{ondrej.kincl@unitn.it}
		\cortext[cor1]{Corresponding author}
		
		\author[unitn]{Ilya Peshkov}
		\ead{ilya.peshkov@unitn.it}
		
		\author[usb,dmi]{Walter Boscheri}
		\ead{walter.boscheri@univ-smb.fr}
		
		\address[unitn]{Laboratory of Applied Mathematics DICAM, University of Trento, 38123 Trento, Italy}
		
		\address[usb]{Laboratoire de Mathématiques UMR 5127 CNRS, Universit{\'e} Savoie Mont Blanc, 73376 Le Bourget du Lac, France}
		
		\address[dmi]{Department of Mathematics and Computer Science, University of Ferrara, 44121 Ferrara, Italy}
		
\begin{abstract}
		This paper contributes to the recent investigations of Lagrangian methods based on Voronoi meshes. The aim is to design a new conservative numerical scheme that can simulate complex flows and multi-phase problems with more accuracy than SPH (Smoothed Particle Hydrodynamics) methods but, unlike diffuse interface models on fixed grid topology, does not suffer from the deteriorating quality of the computational grid. The numerical solution is stored at particles, which move with the fluid velocity and also play the role of the generators of the computational mesh, that is efficiently re-constructed at each time step. The main novelty stems from combining a Lagrangian Voronoi scheme with a semi-implicit integrator for compressible flows. This allows to model low-Mach number flows without the extremely stringent stability constraint on the time step and with the correct scaling of numerical viscosity. The implicit linear system for the unknown pressure is obtained by splitting the reversible from the irreversible (viscous) part of the dynamics, and then using entropy conservation of the reversible sub-system to derive an auxiliary elliptic equation. The method, called SILVA (Semi-Implicit Lagrangian Voronoi Approximation), is validated in a variety of test cases that feature diverse Mach numbers, shocks and multi-phase flows. 
\end{abstract}
%
\end{frontmatter}
	
\section{Introduction}
	Lagrangian methods are valuable for managing intricate deformations and material boundaries. The key characteristic of Lagrangian techniques is that the nodes of the computational mesh move along with the material, instead of remaining fixed in an Eulerian reference frame. As a result, the nonlinear convective terms, which are often challenging, do not appear, leading to a set of numerical schemes that show minimal numerical dissipation at contact discontinuities and material interfaces. Additionally, the control volumes keep a constant mass over time, i.e. no mass flux is allowed to cross the cell boundaries, hence providing a suitable framework for the simulation of multi-phase or multi-material complex phenomena \cite{galera2010two}.
	
	Lagrangian numerical methods can be divided into two main categories: mesh-based and particle-based schemes. The first group is characterized by a conservative discretization of the governing equations on collocated or staggered meshes, see e.g. \cite{Despres2005,Despres2009,Maire2009,Maire2010,LoubereShashkov2004,LoubereMaire2013,Kucharik2014}, which are concerned with invalid control volumes induced by an excessive mesh deformation caused by flows with high vorticity or large velocity gradients. Belonging to the second group, an established particle-based Lagrangian method is SPH (Smoothed Particle Hydrodynamics), in which the computational nodes can move freely and the derivatives are approximated by means of smooth kernel functions \cite{monaghan1992smoothed}. The robustness and simplicity of implementation contributed to an indisputable success of SPH, and the method found diverse applications \cite{violeau2016smoothed,violeau2012fluid}. However, SPH has known limitations, such as the difficulty of implementing boundary conditions \cite{macia2011theoretical} and issues with convergence \cite{zhu2015numerical}. Moreover, since nodal connections in SPH are defined only by a distance threshold, and not by a mesh structure, the stencils are typically larger, leading to more expensive computations with relatively denser matrices \cite{violeau2012fluid}.
	
	Lagrangian methods based on a moving Voronoi mesh are an attempt to combine
	the best from both mesh-based and particle-based schemes. Data are not
	stored at cell vertices or barycenters, but at the mesh generating seeds,
	which move in space like particles following the material speed. The Voronoi
	tessellation is generated anew every time step as in reconnection-based
	methods \cite{loubere2010reale}, and it serves as a supplementary
	construction that helps to devise accurate gradient operators and to
	implement boundary conditions. To the best of our knowledge, Lagrangian
	Voronoi methods were firstly developed in the 1980s \cite{trease1981two,
	trease1988three, Peskin1987}, but then they have been a bit forgotten in
	favor of level-set methods and Arbitrary-Lagrangian-Eulerian schemes. The
	topic was revisited relatively recently by Springel \cite{springel2010pur,
	springel2011hydrodynamic}, Després \cite{despres2024lagrangian} and
	Fernandez \cite{fernandez2018hybrid} among others. Compared to the
	aforementioned publications, our numerical scheme relies on a semi-implicit
	time integrator. More specifically, an explicit discretization is retained
	for those terms of the governing equations which handle slow scale
	phenomena, while an implicit treatment is carried out for the operators
	modeling fast scale dynamics. This makes the scheme particularly well suited
	to deal with low-Mach flows because the time step is not limited by a stiff
	stability condition dictated by the sound speed. Furthermore, an implicit
	treatment of the terms related to the pressure allows for a correct scaling
	of numerical diffusion in the low Mach regime \cite{Dellacherie1}. The class
	of semi-implicit time integrators has proven to be an effective tool for
	addressing complex multi-scale physical problems, see. e.g.
	\cite{Klein,Casulli1990,ParkMunz2005}, along with the family of
	implicit-explicit time stepping techniques
	\cite{AscRuuSpi,PR_IMEX,BosPar2021}. These numerical methods are also
	consistent with the low Mach asymptotic limit of the mathematical model at
	the discrete level, hence they are asymptotic preserving. Here, we explore
	this type of time discretization in the Lagrangian framework. Semi-implicit
	Lagrangian schemes have been proposed in \cite{LGPR} for dealing with stiff
	sources, and recently in \cite{plessier2023implicit} an implicit scheme for
	one-dimensional Lagrangian hydrodynamics has been designed and analyzed. In
	this paper, we aim at providing a semi-implicit discretization to address
	the stiffness of the problem embedded in the pressure fluxes in order to
	develop an all Mach solver for compressible multi-phase flows on
	two-dimensional unstructured meshes. In addition to the asymptotic
	preserving property, a secondary benefit concerning computational efficiency
	in the Lagrangian framework is a better balance between the mesh generation
	and the update of physical variables. In other words, since the time steps
	of a semi-implicit scheme are larger and more expensive, the Voronoi mesh
	needs to be recomputed less frequently. Indeed, in our scheme, the current
	bottleneck is solving a linear system for pressure and not the regeneration
	of the mesh. On the other hand, in a fully explicit discretization, the time
	steps are small but the physical update is cheap. Consequently, the mesh
	generation is a relatively wasteful process.  
	
	This paper follows on a previous publication \cite{SILVAinc}, which
	introduces the SILVA (Semi-Implicit Lagrangian Voronoi Approximation) for
	incompressible flows. This work extends the numerical scheme to compressible
	flows. It is based on a conservative set of equations, which are also
	entropy-stable, namely the entropy is preserved for smooth inviscid flows
	and provably increasing for viscous flows. The time marching scheme relies
	on an operator splitting technique which yields an auxiliary pressure
	equation, resulting in a linear system with a sparse positive-definite
	matrix. The linear system can therefore be efficiently solved at the aid of
	the conjugate gradient method. A modification of the system is proposed to
	fight the issue of odd-even decoupling. The mesh is stabilized by means of
	(continuous) Lloyd iterations and an artificial viscosity is considered to
	deal with potential under-resolution and shocks.
	
	The paper is organized as follows. In Section \ref{sec:tesselation}, we remind the definition and some basic properties of a Voronoi tessellation. This will provide all necessary tools to define the discrete compressible Navier-Stokes system in Section \ref{sec:NS}. This scheme will be, in fact, only semi-discrete because the time is still treated as a continuum and time derivatives appear in the formulation. The semi-implicit discretization of time is revealed subsequently in Section \ref{sec:time}. This is the main focus of the paper in terms of novelty. Section \ref{sec:bcs} discusses how various boundary conditions can be implemented. Finally, Section \ref{sec:results} contains a suite of numerical tests that serve as a validation of the proposed method in terms of robustness and accuracy.  
	
	\section{Voronoi tessellation} \label{sec:tesselation}
	\subsection{Basic properties and generation}
	Let the two-dimensional computational domain $\Omega$ be a convex polygon (this would be a convex polytope in $\mathbb{R}^3$). The Voronoi mesh is defined by a finite set of points $\xx_i \in \Omega, i = 1\dots N$, which are assumed to be distinct and which we call \emph{generating seeds}. A cell $\omega_i$ is associated to each generating seed. The cell is defined as a locus of points in $\Omega$, whose nearest generating seed is $\bm{x}_i$. Mathematically, this can be stated as follows:
	\begin{equation}
		\omega_i = \bigcap_{j \neq i} \bigg\{ \xx \in \Omega: \quad |\xx - \xx_i| < |\xx - \xx_j| \bigg\}, \qquad i,j = 1\dots N.
	\end{equation} 
	The Voronoi tessellation has three essential properties: 
	\begin{enumerate}
		\item Voronoi cells are convex.
		\item The edge $\Gamma_{ij} = \partial \Omega_i \cap \partial \Omega_j$ is a subset of the equidistant line between $\xx_i$ and $\xx_j$. In particular,
		\begin{equation}
			\bm{n}_{ij} = \frac{\bm{x}_j - \bm{x}_i}{|\bm{x}_j - \bm{x}_i|}
		\end{equation}
		defines the outward point normal of cell $\omega_i$ at $\Gamma_{ij}$.
		\item The Voronoi mesh is a topological dual of Delaunay triangulation \cite{chynoweth1990mesh} (except for certain singular cases, like a Cartesian grid) .
	\end{enumerate}
	
	There are many papers devoted to the study of Voronoi grids, and several open source libraries with Voronoi mesh generation functionality are available, such as Voro++ \cite{rycroft2009voro++}, CGAL \cite{fabri2009cgal} and Qhull \cite{barber2013qhull}. These libraries rely on constructing the Delaunay triangulation first, and then building the Voronoi mesh by connecting the centers of circumcircles. The benefit of this approach is the ability to handle even pathological distribution of points. In our case, however, the mesh is typically very regular and the distribution of generating seeds is almost uniform. Therefore, it is advantageous to use a direct method. Direct methods involve the following two steps. Firstly, the generating seeds are arranged in a neighbor list or an octree. Then, starting from $\omega_i = \Omega$, the $i$-th cell $\omega_i$ is cut by intersecting it with the half-plane $\{|\xx - \xx_i| < |\xx - \xx_j|\}$ for every $j$, prioritizing seeds in the vicinity of $\xx_j$. The process has a suitable terminating condition, such that only nearby generating seeds need to be queried. The direct approach is preferred in the context of Lagrangian Voronoi methods because it is very fast and it can be easily parallelized \cite{ray2018meshless, liu2020parallel}. 
	
	Throughout this manuscript, we will use the following notation. Let $\xx_{ij} = \xx_i - \xx_j$, and let $|\omega_i|$ and $|\Gamma_{ij}|$ be the area of cell $\omega_i$ and the length of edge $\Gamma_{ij}$, respectively. By
	\begin{equation}
		\bm{c}_i := \frac{1}{|\omega_i|} \int_{\omega_i} \xx \; \dd \xx
	\end{equation}
	we denote the \textit{centroid} of the cell $\omega_i$ and by
	\begin{equation}
		\bm{m}_{ij} := \frac{1}{|\Gamma_{ij}|} \int_{\Gamma_{ij}} \xx \; \dd S(\xx)
	\end{equation}
	we refer to the midpoint of the edge $|\Gamma_{ij}|$. Note that $\bm{m}_{ij}$ is generally different from 
	\begin{equation}
		\oxx_{ij} := \frac{\xx_i + \xx_j}{2},
	\end{equation}
	see Figure \ref{fig:cell} (it is even possible that $\oxx_{ij} \notin \Gamma_{ij}$!). The cells will serve as the computational nodes for our simulations. We also introduce the inter-seeds distance $r_{ij}=|\xx_{ij}|$. As a characteristic size of the spatial discretization we use 
	\begin{equation}
		\delta r = \min \sqrt{|\omega_i|}
	\end{equation}
which would correspond to a side length on a Cartesian mesh (the third root is needed in three dimensional space). For every variable $f$ (such as density $\rho$, velocity $\vv$ or specific internal energy $e_i$), $f_i$ represents its point value at $\xx_i$. We also use the following definitions for difference and mean values, which are common in the SPH literature:
	\begin{equation}
		f_{ij} := f_i - f_j,
	\end{equation}
	and
	\begin{equation}
		\overline{f}_{ij} = \frac{f_i + f_j}{2}.
	\end{equation}
	
	\begin{figure}
		\centering
		\includegraphics[width=0.3\linewidth]{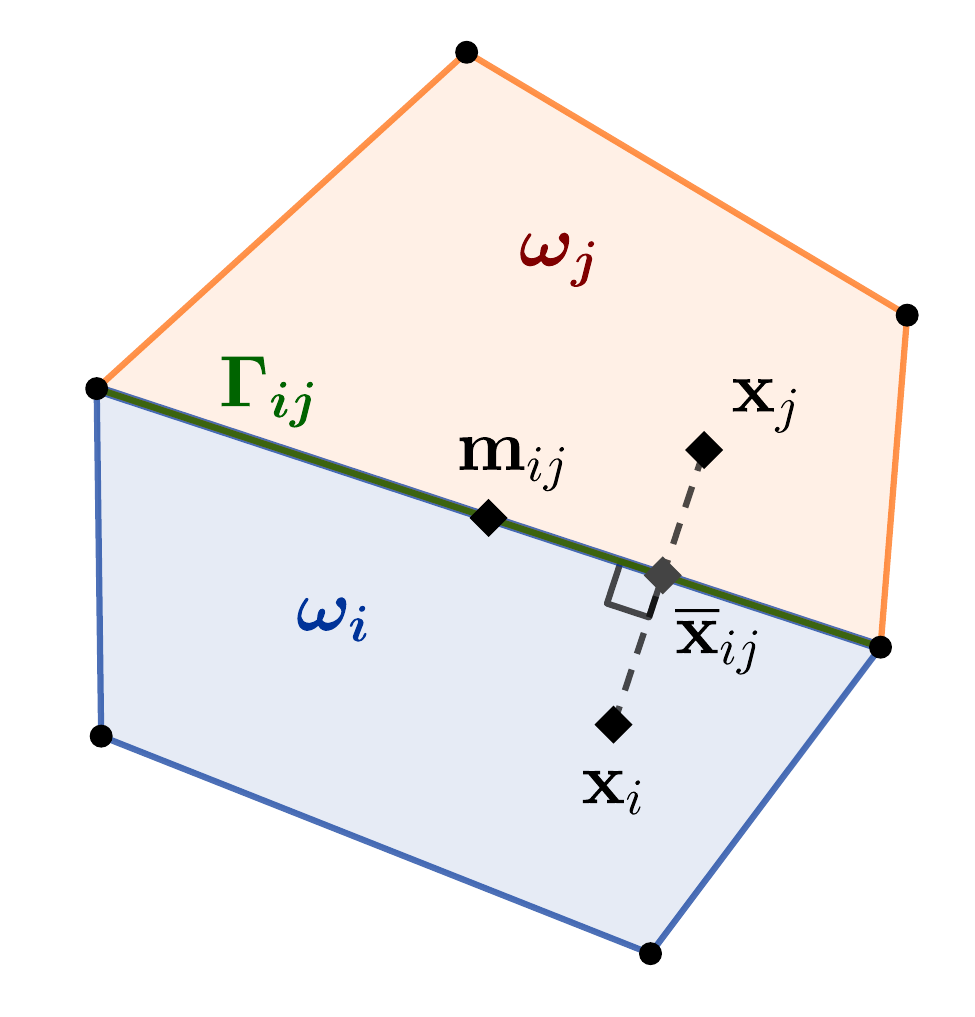}
		\caption{Two Voronoi cells $\omega_i$ and $\omega_j$ with our notation for edges, seed generators and edge midpoints.}
		\label{fig:cell}
	\end{figure}
	
	\subsection{Discrete derivative operators on Voronoi tessellations} \label{sec.derivatives}
	To approximate the gradient of a scalar variable $f$ at the generating seed locations, we combine an integral identity:
	\begin{equation}
		\mathbb{I} = \frac{1}{|\omega_i|} \int_{\partial \omega_i}  (\xx - \xx_i) \otimes \bm{n} \, \dd{S},
	\end{equation}
	(which holds for every cell $\omega_i$ by virtue of divergence theorem) and an approximation
	\begin{equation}
		\nabla f(\xx_{i})\cdot \xx_{ij} \approx f_{ij},
	\end{equation}
	(which is exact for polynomials of first degree). Then, for any interior cell $\omega_i$, the gradient of the function $f$ is evaluated as
	\begin{align}
		\nabla f(\xx_i) &= \frac{1}{|\omega_i|} \int_{\partial \omega_i}  (\xx - \xx_i) \; \left(\nabla f(\xx_i) \cdot \bm{n} \right) \, \dd{S} \nonumber\\
		&= \frac{1}{|\omega_i|} \sum_{j \neq i}  \int_{\Gamma_{ij}}  (\xx - \xx_i) \; \left(\nabla f(\xx_i) \cdot \bm{n}_{ij} \right) \, \dd{S} \nonumber\\
		&= -\frac{1}{|\omega_i|} \sum_j \frac{|\Gamma_{ij}|}{r_{ij}}(\bm{m}_{ij} - \xx_i) \left(\nabla f(\xx_i) \cdot \xx_{ij} \right)
	\end{align}
	yielding the following definition of the discrete gradient operator:
	\begin{equation}
		\sgrad{f}{i} := -\frac{1}{|\omega_i|}\sum_j \frac{|\Gamma_{ij}|}{r_{ij}} f_{ij}(\bm{m}_{ij} - \xx_i)
		\label{eq:sgrad}.
	\end{equation}
	The analysis of consistency of the above operator, including a simple error estimate, can be found in \cite{springel2010pur}.
	We shall use \eqref{eq:sgrad} regardless of whether the cell $\omega_i$ is an interior cell or a boundary cell. By this notion, we implicitly assume a homogeneous Neumann boundary condition. We will postpone the discussion of boundary conditions in Section \ref{sec:bcs}, imagining for now a periodic domain.
	
    Note that \eqref{eq:sgrad} approximates the point value of gradient, rather than the cell average. In this sense, a Lagrangian Voronoi method based on \eqref{eq:sgrad} is closer to SPH or FDM (finite difference method) than to FVM (finite volume method). Formula \eqref{eq:sgrad} can be extended to divergence and vector gradient operators by pure algebra. For instance, the discrete  divergence operator for a vector field $\bm{U}$ is defined as
	\begin{equation}
		\sdiv{\bm{U}}{i} := -\frac{1}{|\omega_i|}\sum_j \frac{|\Gamma_{ij}|}{r_{ij}} \bm{U}_{ij} \cdot (\bm{m}_{ij} - \xx_i)
		\label{eq:sdiv},
	\end{equation}
	the gradient of a vector field is approximated by 
	\begin{equation}
		\sgrad{\bm{U}}{i} := -\frac{1}{|\omega_i|}\sum_j \frac{|\Gamma_{ij}|}{r_{ij}} \bm{U}_{ij} \otimes (\bm{m}_{ij} - \xx_i),
		\label{eq:sgrad_vector}
	\end{equation}
	and for a divergence of a tensor field $\mathbb{T}$, we write
	\begin{equation}
		\sdiv{\mathbb{T}}{i} := -\frac{1}{|\omega_i|}\sum_j \frac{|\Gamma_{ij}|}{r_{ij}} \mathbb{T}_{ij}(\bm{m}_{ij} - \xx_i)
		\label{eq:sdiv_tensor}.
	\end{equation}
	
	Furthermore, we can define the approximate $L^2$ scalar product:
	\begin{equation}
		(f, g) := \sum_i |\omega_i| f_i \, g_i.
	\end{equation}
	In analogy to SPH, the scalar product allows to introduce another gradient approximation, which is the negative adjoint of $\sgrad{f}{}$. In other words, for any scalar field $g$, we define $\adgrad{g}{}$ implicitly such that
	\begin{equation}
		\sum_i |\omega_i| \sgrad{f}{i} g_i = - \sum_i |\omega_i| f_i \adgrad{g}{i}. \label{eq:byparts}
	\end{equation}
	After some algebra, it is possible to find an explicit formula for the negative adjoint operator:
	\begin{equation}
		\adgrad{f}{i} := \frac{1}{|\omega_i|} \sum_j \frac{|\Gamma_{ij}|}{r_{ij}} \bigg( f_{ij}(\bm{m}_{ij} - \oxx_{ij}) -  \overline{f}_{ij} \xx_{ij} \bigg).
		\label{eq:adgrad}
	\end{equation}
	The operator $\adgrad{g}{}$ embodies a weak approximation of the gradient, in the sense that \eqref{eq:byparts} mimics the integration by parts. 
	The negative adjoint operator \eqref{eq:adgrad} appears naturally in the volume differential formula
	\begin{equation}
		\frac{\dd|\omega_i|}{|\omega_i|} = \addiv{\dd \xx}{i} \label{eq:domega},
	\end{equation}
	which quantifies how cell areas are changing when the seed generators vary. The above formula is related to the Reynolds transport theorem, which may be written in the form:
	\begin{equation}
		\dv{}{t}\int_{\omega_i}\dd \xx = \int_{\omega_i} \divg{\bm{U}} \dd \xx \label{eq:rtt},
	\end{equation}
	where $\bm{U}$ is an advective velocity field such that $\omega_i$ is a material cell. The exact proof is, however, not easy, and we refer the interested reader to \cite{SILVAinc} or \cite{despres2024lagrangian}.

	\section{Compressible Navier-Stokes system on a moving Voronoi mesh} \label{sec:NS}
	Let us assume that $\Omega$ is a periodic computational domain and $\omega_i$ are periodic Voronoi cells. In convective form, the compressible Navier-Stokes equations write
	\begin{align}
		\dv{\rho}{t} =&  -\rho \divg{\vv}, \\
		\rho \dv{\vv}{t}  =&  \divg{\mathbb{T}}, \\
		\rho \dv{e}{t}  =& \divg{(\mathbb{T}\bm{v})}, 
	\end{align}
	where the variables of the system are $\rho, \vv, e$ (density, velocity and specific total energy), $\mathbb{T}$ is a stress tensor and
	\begin{equation}
		\dv{f}{t} = \pdv{f}{t} + \vv \cdot \nabla f
	\end{equation}
	is the material (or Lagrangian) time derivative. Aiming at designing a Lagrangian method, we discretize these equations for the Voronoi generating seeds $\bm{x}_i$ which travel with velocity
	\begin{equation}
		\dv{\xx_i}{t} = \vv_i.
	\end{equation}
	Using the discrete differential operators presented in Section \ref{sec.derivatives}, we write a preliminary semi-discrete scheme as follows:
	\begin{align}
		\dv{\xx_i}{t} &= \; \vv_i, \label{eq:path}\\
		\dv{\rho_i}{t} &=  -\rho_i \Tr \mathbb{L}_i , \label{eq:bomass}\\
		\rho_i \dv{\vv_i}{t}  &=  \bm{f}_i, \label{eq:bomom}\\
		\rho_i \dv{e_i}{t}  &= \bm{f}_i \cdot \vv_i + \mathbb{L}_i:\mathbb{T}_i, \label{eq:boe}
	\end{align}
	with the definitions
	\begin{align}
		\bm{f}_i &:= \langle \divg \mathbb{T} \rangle_i, \\ 
		\mathbb{L}_i &:= \adgrad{\vv}{i}.
	\end{align}
	At all times, it is assumed that $\mathbb{T}_i$ is symmetric. More precisely, the stress tensor is usually described as a summation of a pressure contribution and a viscous part:
	\begin{equation}
		\mathbb{T}_i = -p_i \mathbb{I} + \mathbb{S}_i,
	\end{equation}
	where $p_i = p(\rho_i, \vv_i, e_i)$ is the pressure and, under Stokes hypothesis, the viscous stress tensor writes
	\begin{equation}
		\mathbb{S}_i = 2(\mu_i + \mu^\text{art}_i) \left( \mathbb{D}_i - \frac{1}{3} (\Tr \mathbb{D}_i) \mathbb{I}\right),
	\end{equation}
	where $\mu_i \leq 0$ is the dynamic viscosity coefficient and 
	\begin{equation}
		\mathbb{D}_i = \frac{1}{2}\left( \mathbb{L}_i + \mathbb{L}_i^T\right)
	\end{equation}
	is the velocity deformation. To account for shocks, we use the Stone and
	Norman viscous tensor \cite{stone1992zeus}, which is a multi-dimensional
	extension of von Neumann and Richtmyer artificial viscosity
	\cite{margolin_richtmyer_2024}:
	\begin{equation}
		\mu^\text{art}_i =  \begin{cases}
			-(\delta r)^2 \rho_i (\Tr \mathbb{D}_i), & \Tr \mathbb{D}_i < 0\\
			0, & \Tr \mathbb{D}_i \geq 0
		\end{cases},
		\label{eq:artificial_visc}
	\end{equation}
	Broadly speaking, the idea of expression \eqref{eq:artificial_visc} is to increase viscosity near shocks, characterized by compression of the fluid ($\Tr \mathbb{D}_i < 0$), in such a way that is notable only in the under-resolved regions of the simulation (hence the factor $h^2$).

	The following two theorems reflect on the conservative properties and on the discrete entropy inequality satisfied by the semi-discrete scheme \eqref{eq:path}-\eqref{eq:boe}.
	
	\begin{theorem}[Discrete conservation laws]
		Let the computational domain $\Omega \in \mathbb{R}^3$ be assigned with periodic boundaries. Let $\xx_i$, $\rho_i$, $\vv_i$, $e_i$ obey \eqref{eq:path}-\eqref{eq:boe}. Define
		\begin{equation}
			\begin{split}
				M_i &:= |\omega_i| \rho_i,\\
				\bm{U}_i &:= |\omega_i| \rho_i \vv_i = M_i \vv_i,\\
				E_i &:= |\omega_i| \rho_i e_i = M_i e_i.
			\end{split}
		\end{equation}
		(These variables represent the total mass, momentum and energy of each cell respectively.) 
		Then, the masses $M_i$ are constant and moreover, it holds true that
		\begin{align}
			\dv{}{t}\sum_i \bm{U}_i &= \dv{}{t}\sum_i \xx_i \times \bm{U}_i = 0\\
			\dv{}{t} \sum_i E_i &= 0. \label{conservation}
		\end{align}
		In other words, the total momentum (both linear and angular) and energy are integrals of motion.
	\end{theorem}
	\begin{proof}
		The mass conservation is an immediate consequence of \eqref{eq:domega} and the product rule. Using the discrete integration by parts \eqref{eq:byparts}, we find
		\begin{equation}
			\dv{E_i}{t} = \sum_i |\omega_i| \left( \ff_i \cdot \bm{v}_i + \mathbb{L}_i : \mathbb{T}_i \right) = 0.
		\end{equation}
		Let the vector $\bm{c} \in \mathbb{R}^3$ be arbitrary. Then
		\begin{equation}
			\dv{}{t}\sum_i \bm{c} \cdot \bm{U}_i = \sum_i |\omega_i| \ff_i \cdot \bm{c} = -\sum_i |\omega_i| \mathbb{T}_i : \adgrad{\bm{c}}{i} = 0,
		\end{equation}
		since 
		\begin{equation}
			\adgrad{\bm{c}}{i} = -\bm{c} \cdot \sum_j \frac{|\Gamma_{ij}|}{r_{ij}} \xx_{ij} = \int_{\partial \omega_i} \bm{n} = 0.
		\end{equation} 
		Without loss of generality, substituting
		\begin{equation}
			\bm{c} \in \left\{ \begin{pmatrix}
				1\\
				0\\
				0
			\end{pmatrix}, 
			\begin{pmatrix}
				0\\
				1\\
				0
			\end{pmatrix},
			\begin{pmatrix}
				0\\
				0\\
				1
			\end{pmatrix}\right\}
		\end{equation}
		proves the conservation of momentum in $x,y$ and $z$ direction. Finally, let $\Theta \in \mathbb{R}^{3\times 3}$ be an arbitrary skew-symmetric matrix. Then
		\begin{equation}
			\dv{}{t}\sum_i \bm{U}_i \cdot \Theta \xx_i =  \sum_i M_i\vv_i \cdot \Theta \vv_i + \sum_i |\omega_i| \bm{f}_i\cdot \Theta \xx_i . 
		\end{equation}
		The first term on the right hand side vanishes, since skew-symmetric quadratic forms are alternating. Invoking the discrete integration by parts \eqref{eq:byparts} and the exactness of the discrete gradient for linear functions, we have that
		\begin{equation}
			\sum_i |\omega_i| \bm{f}_i\cdot \Theta \xx_i = -\sum_i  |\omega_i| \mathbb{T}: \sgrad{(\Theta \xx)}{i} =  -\sum_i  |\omega_i| \mathbb{T}: \Theta_i.
		\end{equation}
		This is zero by virtue of orthogonality between symmetric and skew-symmetric matrices. Without loss of generality, substituting  
		\begin{equation}
			\Theta \in \left\{ 
			\begin{pmatrix}
				0 & 0 & 0\\
				0 & -1 & 0\\
				1 & 0 & 0
			\end{pmatrix},
			\begin{pmatrix}
				0 & 0 & 1\\
				0 & 0 & 0\\
				-1 & 0 & 0
			\end{pmatrix},
			\begin{pmatrix}
				0 & -1 & 0\\
				1 & 0 & 0\\
				0 & 0 & 0
			\end{pmatrix},
			\right\}
		\end{equation}
		proves the conservation of angular momentum in $x,y$ and $z$ direction.
	\end{proof}

	\begin{theorem}[Discrete entropic inequality]\label{th.entropy}
		Let 
		\begin{align}
			\eta_i &= \eta(\rho_i, \vv_i, e_i),\\
			T_i &= T(\rho_i, \vv_i, e_i),
		\end{align}
		be the specific entropy and temperature, respectively, such that the first law of thermodynamics
		\begin{equation}
			\dd e_i = \vv_i \cdot \dd \vv_i + T_i \dd \eta_i + \frac{p_i}{\rho_i^2}\dd \rho_i \label{eq:first_law}
		\end{equation}
		is satisfied, and let \eqref{eq:path}-\eqref{eq:boe} hold true. Then 
		\begin{equation}
			\dv{\eta_i}{t} = \frac{2}{\rho_i T_i} (\mu_i + \mu_i^\mathrm{art})\|\mathbb{D}_i\|_F^2 \geq 0,
		\end{equation}
		where $ \|\mathbb{D}_i\|_F = \sqrt{\mathbb{D}_i : \mathbb{D}_i}$ denotes the Frobenius norm.
	\end{theorem}
	\begin{proof}
		Using \eqref{eq:first_law} and the discrete balance laws \eqref{eq:bomass}-\eqref{eq:boe} justifies the following computation:
		\begin{equation}
			\begin{split}
				\rho_i T_i \dv{\eta_i}{t} &=  \rho_i \dv{e_i}{t} - \rho_i \vv_i \cdot \dv{ \vv_i}{t}  - \frac{p_i}{\rho_i} \dd \rho_i\\ 
				&= \bm{f}_i \cdot \vv_i + \mathbb{L}_i : \mathbb{T}_i - \bm{f}_i \cdot \vv_i + p_i \Tr \mathbb{L}_i \\
				&= \mathbb{S}_i : \mathbb{L}_i = \mathbb{S}_i : \mathbb{D}_i = 2(\mu_i + \mu_i^\mathrm{art})\|\mathbb{D}_i\|_F^2.
			\end{split}
		\end{equation}
	\end{proof}
	
	We note that this does not guarantee entropy conservation up to machine accuracy in the final scheme, especially because of errors introduced by inexact time integration (as opposed to the conservation of mass, energy and momentum).
	
	\section{Time marching scheme} \label{sec:time}
	The time interval is defined as $t\in[0,t_\mathrm{end}]$, with $t_\mathrm{end}$ being the final time. A sequence of time steps $\dt=t^{n+1}-t^{n}$, where $n$ denotes the current time level index, are used to discretize the temporal domain. Let us rely on the Lie-Trotter splitting scheme, thus dividing the right hand-side of \eqref{eq:path}-\eqref{eq:boe} into time reversible and irreversible sub-systems:
	\begin{equation}
		\dv{\bm{X}}{t} = F_\mathrm{rev}(\bm{X}) + F_\mathrm{irr}(\bm{X}), 
		\label{eq:odes_abstract}
	\end{equation}
	where
	\begin{equation}
		\bm{X} = \begin{pmatrix}
			\xx_i \\
			\rho_i \\
			\vv_i \\
			e_i \\
		\end{pmatrix}
	\end{equation}
	is our "phase space vector". $F_\mathrm{rev}$ encompasses all reversible terms, and $F_\mathrm{irr}$ is the remainder (namely the viscous force and the work exerted by viscosity). Considering the flow maps:
	\begin{align}
		\bm{X}^1 &= \phi_{\mathrm{rev}}(\bm{X}^0, \dt) \iff \left( \dv{\bm{X}}{t} =  F_\mathrm{rev}(\bm{X}), \; \bm{X}^1 = \bm{X}(\dt), \; \bm{X}^0 = \bm{X}(0) \right),
		\\
		\bm{X}^1 &= \phi_{\mathrm{irr}}(\bm{X}^0, \dt) \iff \left( \dv{\bm{X}}{t} =  F_\mathrm{irr}(\bm{X}), \; \bm{X}^1 = \bm{X}(\dt),\; \bm{X}^0 = \bm{X}(0)
		\right),
	\end{align}
	the Lie-Trotter splitting scheme estimates the solution of \eqref{eq:odes_abstract} by:
	\begin{equation}
		\bm{X}(t + \dt) \approx \phi_{\mathrm{irr}}(\phi_{\mathrm{rev}}(\bm{X}(t), \dt), \dt).
	\end{equation}
	The scheme is first order in time and, in our case, the flow maps will themselves be approximated at a first order. 
	
	The adoption of the operator splitting approach is advantageous because $\phi_{\mathrm{rev}}$ has the additional property of being adiabatic (preserving the specific entropy of each cell), allowing us to write a simpler system for the pressure.
	
	\subsection{Semi-implicit reversible step} \label{sec:reversible}
	The reversible sub-system writes
	\begin{align}
		\dv{\xx_i}{t} &= \; \vv_i,\\
		\dv{\rho_i}{t} &=  -\rho_i \addiv{\vv}{i}, \label{eq:bomass_rev}\\
		\rho_i \dv{\vv_i}{t}  &=  -\sgrad{p}{i}, \label{eq:bomom_rev}\\
		\rho_i \dv{e_i}{t}  &= -\sgrad{p}{i} \cdot \vv_i - p_i \addiv{\vv}{i}. \label{eq:boe_rev}
	\end{align}
	This sub-system requires some implicitness to account for the stiff coupling between density and pressure in low Mach flows. In light of Theorem \ref{th.entropy}, we see that the specific entropy of every cell is conserved by this sub-system. Therefore, we may add an auxiliary equation for the pressure
	\begin{equation}
		\dv{p_i}{t} = -\rho_i c_i^2 \addiv{\vv}{i},
	\end{equation}
	where
	\begin{equation}
		c = \sqrt{\left( \pdv{p}{\rho} \right)_\eta}
	\end{equation}
	is the speed of sound. Naturally, this equation is by no means independent from the rest of the system. Still, we can use it to obtain a prediction for the pressure at the future time level $t^{n+1}$ in the following semi-implicit scheme:
	\begin{align}
		\frac{\xx_i^{n+1} - \xx_i^{n}}{\dt} &= \vv_i^{n}, \label{eq:path_implicit}\\
		\rho_i^{n+1} &= \frac{M_i}{|\omega_i|^{n+1}}, \label{eq:bomass_implicit}\\
		\frac{\tilde{\vv}_i^{n+1} - \vv_i^{n}}{\dt}  &=  -\frac{1}{\rho_i^{n+1}}\sgrad{q}{i}^{n+1}, \label{eq:bomom_implicit}\\
		\frac{\tilde{e}_i^{n+1} - e_i^{n}}{\dt}  &= -\frac{1}{\rho_i^{n+1}}\left( \sgrad{q}{i}^{n+1} \cdot \tilde{\vv}_i^{n+1} + q_i^{n+1} \addiv{\tilde{\vv}}{i}^{n+1} \right),  \label{eq:boe_implicit}\\
		\frac{q_i^{n+1} - p_i^{n}}{\dt \, (c_i^{n})^2} &= -\rho_i  \addiv{\tilde{\vv}}{i}^{n+1} \label{eq:q_implicit}.
	\end{align}
     We denote the updated pressure by $q_i^{n+1}$ to distinguish it from 
	\begin{equation}
		p_i^{n+1} = p(\rho_i^{n+1},\vv_i^{n+1}, e_i^{n+1}),
	\end{equation}
	which is separated from $q_i^{n+1}$ by a time-discretization error. The tilde symbol used for $\tilde{\vv}_i^{n+1}$ and $\tilde{e}_i^{n+1}$ indicates that these are provisional values, since we still need to add the explicit \textit{irreversible} terms to obtain the velocity and internal energy at the next time level. Note that there is no equation of continuity and we rely instead on mass conservation to update density. Indeed, using the continuity equation to update $\rho$ would violate the conservation laws. Also note that the gradient and divergence operator depend on $\xx$ through the Voronoi tessellation. It is assumed that the updated Voronoi mesh defined by seeds $\xx_i$ is used. Equations \eqref{eq:path_implicit} and \eqref{eq:bomass_implicit} explicitly determine $\xx_i^{n+1}$ and $\rho_i^{n+1}$. In order to untangle the implicit part of the problem, it is now sufficient to solve for $q_i^{n+1}$. To this end, formal substitution of \eqref{eq:bomom_implicit} into \eqref{eq:q_implicit} leads to the linear system
	\begin{equation}
		|\omega_i|^{n+1} \left( \frac{q_i^{n+1} }{\rho_i(c^{n}_i \dt)^2} -  \bigg \langle \divg^* \frac{\langle \nabla q \rangle }{\rho}\bigg \rangle_{i}^{n+1} \right)=|\omega_i|^{n+1}\left( \frac{p^{n}_i }{\rho_i(c^{n}_i \dt)^2} -\frac{1}{\dt} \addiv{\vv^{n}}{i} \right). \label{eq:linprob1}
	\end{equation}
	To analyze the symmetric property of the above system, let us express the left hand side of \eqref{eq:linprob1} compactly as
	\begin{equation}
		\sum_j A_{ij} \, q_j^{n+1},
	\end{equation}
    with the system matrix $\mathbb{A}=\{A_{ij}\}$. Observing that
	\begin{equation}
		\sum_i \phi_i A_{ij}\theta_j = \sum_i \frac{|\omega_i|^{n+1}}{\rho_i^{n+1}} \left( \frac{\phi_i \theta_i}{(c^{n}_i \dt)^2} + \sgrad{\phi}{i} \cdot \sgrad{\theta}{i}\right),
	\end{equation}
	for any $\phi, \theta$, it is clear that $\mathbb{A}$ is a symmetric positive-definite matrix.
	
	The speed of sound $c^{n}_i$ and pressure $p^{n}_i$ at the current time level be must computed from the equation of state. For example, for an ideal gas we have that
	\begin{align}
		p &= (\gamma - 1) \rho \bigg(e - \frac{v^2}{2}\bigg), \\
		c &= \sqrt{\frac{\gamma p}{\rho}},
	\end{align}
	($\gamma$ is the adiabatic index of the fluid), but a different equation of state can be used. Note that in the incompressible limit ($c \to \infty$), the linear system \eqref{eq:linprob1} simplifies to a Poisson problem
	\begin{equation}
		- \frac{1}{\rho_i} \bigg \langle {\divg}^* {\langle \nabla q\rangle }\bigg \rangle_{i}^{n+1}= -\frac{1}{\dt} \addiv{\vv}{i}^{n+1}, \label{eq:incompressible}
	\end{equation}
	where $q$ is no longer related to the equation of state and is merely a Lagrange multiplier enforcing the incompressibility constraint. An equation very similar to \eqref{eq:incompressible} has been obtained in the incompressible version of SILVA \cite{SILVAinc}. The energy is needed only in the evaluation of $p^{n}_i$ and $c^{n}_i$, and in the low Mach limit, the balance of momentum becomes decoupled from the energy equation. On the contrary, in the high Mach number regime, $q_i^{n+1}$ does not differ significantly from $p_i^{n+1}$, yielding a scheme that is nearly explicit and where the energy balance is more important.
	
	Nonetheless, there are reasons to avoid using $A_{ij}$ directly in a numerical code. One significant disadvantage is that $A_{ij}$ operates on a relatively large stencil, as $A_{ij} \neq 0$	not only when $\omega_i$ and $\omega_j$ are neighboring cells, but also when they share a neighbor. Thus, the matrix $\mathbb{A}$ is not very sparse. Yet, there is a second and more severe drawback. Our gradient and divergence operator are essentially a generalization of central differences. In one space dimension, it is well known that the approximation of second order derivatives by means of central finite differences leads to the decoupling of degrees of freedom associated with odd and even nodes. This numerical artifact, called \textit{odd-even decoupling}, becomes only worse in higher dimensions. The problem persists on unstructured grids, where strong interaction with the second layer of neighborhood reinforces the noise with a characteristic checkerboard pattern \cite{date2003fluid}. 
	This pathological behavior can be suppressed using 
	\begin{equation}
		\sgrad{q}{i}^{n+1} \cdot \xx_{ij}^{n+1} \approx q_{ij}^{n+1} \approx  \sgrad{q}{j}^{n+1} \cdot \xx_{ij}^{n+1},
	\end{equation}
	which holds exactly for polynomials of first degree. Application of the above approximation to the second order derivative in \eqref{eq:linprob1} yields
	\begin{equation}
		\begin{split}
			-|\omega_i|^{n+1}\bigg \langle {\divg}^* { \frac{\langle \nabla q\rangle}{\rho} }\bigg \rangle_{i}^{n+1} &= \sum_j \frac{|\Gamma_{ij}|^{n+1}}{r_{ij}^{n+1}} \bigg(   \left( \frac{\sgrad{q}{i}^{n+1}}{2\rho_i^{n+1}} + \frac{\sgrad{q}{j}^{n+1}}{2\rho_j^{n+1}} \right)\cdot \xx_{ij}^{n+1}- \left( \frac{\sgrad{q}{i}^{n+1}}{\rho_i^{n+1}} - \frac{\sgrad{q}{j}^{n+1}}{\rho_j^{n+1}} \right)\cdot (\bm{m}_{ij}^{n+1} - \oxx_{ij}^{n+1}) \bigg) \\
			&\approx \sum_j \frac{|\Gamma_{ij}|^{n+1}}{r_{ij}^{n+1}} \bigg(   \left( \frac{1}{2\rho_i^{n+1}} + \frac{1}{2\rho_j^{n+1}} \right)q_{ij}^{n+1}- \left( \frac{\sgrad{q}{i}^{n+1}}{\rho_i^{n+1}} - \frac{\sgrad{q}{j}^{n+1}}{\rho_j^{n+1}} \right)\cdot (\bm{m}_{ij}^{n+1} - \oxx_{ij}^{n+1}) \bigg).
		\end{split}
		\label{eq:sparsification_trick}
	\end{equation}
	This trick is similar to how Laplacian is approximated in incompressible SPH \cite{violeau2012fluid}. The improved version of the linear system \eqref{eq:linprob1} is compactly written at the aid of the matrices $\mathbb{B}=\{B_{ij}\}$  and $\mathbb{C}=\{C_{ij}\}$ as
	\begin{equation}
		\sum_j B_{ij} q_i^{n+1} - \sum_j C_{ij} q_i^{n+1}  = b_i^n\label{eq:linprob2}
	\end{equation}
	where
	\begin{align}
		\sum_j B_{ij} q_j^{n+1} &= |\omega_i|^{n+1}\frac{q_i^{n+1}}{\rho_i^{n+1}(c^{n}_i \dt)^2} + \sum_j \frac{|\Gamma_{ij}|^{n+1}}{r_{ij}^{n+1}}  \left( \frac{1}{2\rho_i^{n+1}} + \frac{1}{2\rho_j^{n+1}} \right)q_{ij}^{n+1}, \\
		\sum_j C_{ij} q_j^{n+1} &= \sum_j \frac{|\Gamma_{ij}|^{n+1}}{r_{ij}^{n+1}}   \left( \frac{\sgrad{q}{i}^{n+1}}{\rho_i^{n+1}} - \frac{\sgrad{q}{j}^{n+1}}{\rho_j^{n+1}} \right)\cdot (\bm{m}_{ij}^{n+1} - \oxx_{ij}^{n+1}) ,\\
		b_i &= |\omega_i|^{n+1}\left( \frac{p^{n}_i }{\rho_i^{n+1}(c^{n}_i \dt)^2} -\frac{1}{\dt} \addiv{\vv^{n}}{i} \right).
	\end{align}
	Trivially, the matrix $\mathbb{B}$ is symmetric positive-definite and sparse. The matrix $\mathbb{C}$ represents a correction term for irregular grids which is still not very sparse but it allows to solve \eqref{eq:linprob2} by a fixed point iteration method, namely
	\begin{equation}
		\sum_j B_{ij} q_i^{r+1} = b_i + \sum_j C_{ij} q_i^r,\label{eq:linprob3}
	\end{equation}
    where $r$ denotes the current iteration number, starting with $q^0 = p^n$. System \eqref{eq:linprob3} can be efficiently solved by a matrix-free conjugate gradient method, so there is no need to allocate memory every time step to save a sparse matrix. We tested both variants and found that \eqref{eq:linprob3} is faster than solving the original system \eqref{eq:linprob1}, and it is also more accurate. Once $q_i^{n+1}$ is known, we can compute $\tilde{\vv}_i^{n+1}$ and $\tilde{e}_i^{n+1}$ from \eqref{eq:bomom_implicit} and \eqref{eq:boe_implicit}, respectively.
	
	\subsection{Explicit irreversible step}
	Extraction of the irreversible terms from system \eqref{eq:path}-\eqref{eq:boe} leaves us with the following set of semi-discrete equations:
	\begin{align}
		\rho_i \dv{\vv_i}{t} &= \sdiv{\mathbb{S}}{i}, \label{eq:bomom_irr} \\
		\rho_i \dv{e_i}{t} &=  \sdiv{\mathbb{S}}{i} \cdot \vv_i + \mathbb{L}_i:\mathbb{S}_i.
	\end{align}
	As we are interested in flows with small or negligible viscosity, we choose a simple forward Euler time integrator:
	\begin{align}
		\rho_i \frac{\vv_i^{n+1} - \tilde{\vv}_i^{n+1}}{\dt} &= \sdiv{\mathbb{S}}{i}^n \\
		\rho_i \frac{e_i^{n+1} - \tilde{e}_i^{n+1}}{\dt} &=  \sdiv{\mathbb{S}}{i}^n \cdot \vv^n_i + \mathbb{L}^n_i:\mathbb{S}^n_i.
	\end{align}
	
	This implementation again suffers from the odd-even decoupling problem, which can potentially plague the velocity field with checkerboard oscillations. In case of physical viscosity with constant dynamic coefficient $\mu$, and if moreover bulk viscosity is neglected, the viscous term can be handily replaced by
	\begin{equation}
		\begin{split}
			\sdiv{\mathbb{S}}{i}^n &\approx \mu \langle \Delta \vv^n \rangle_i,
		\end{split}
		\label{eq:viscosity_approx}
	\end{equation}
	where 
	\begin{equation}
		\langle \Delta \vv^n \rangle_i = -\frac{1}{|\omega_i|^{n+1}}\sum_j \frac{|\Gamma_{ij}|^{n+1}}{r_{ij}^{n+1}}\vv_{ij}^n \label{eq:Laplace}
	\end{equation}
	is a discrete approximation of the Laplacian with compact stencil. However, we shall use $\sdiv{\mathbb{S}}{i}^n$ everywhere in this paper because it is obviously more robust, and even in cases where \eqref{eq:viscosity_approx} can be used, we did not find substantial difference neither in the quality of results nor in the performance of the algorithm.
	
	\begin{remark}
		We keep in mind that an explicit discretization of the viscous terms is sub-optimal for flows where viscosity is a dominating force because it is only first order accurate and easily runs into stiffness issues. We leave the design of an implicit viscous solver for future investigations, see e.g. \cite{BT22}. 
	\end{remark}
	
	\subsection{Mesh relaxation step}
	In our simulations, we usually start from a structured Voronoi mesh, for example a hexagonal grid. Needless to say, this regularity is lost very quickly due to the Lagrangian motion of the Voronoi seeds. Although Voronoi meshes are robust by definition, it is a pitfall to think that they always exhibit high quality. Examples of a low quality Voronoi mesh is depicted in Figure \ref{fig:badvoronoi}.
	\begin{figure}[htb!]
		\centering
		\includegraphics[width=0.5\linewidth]{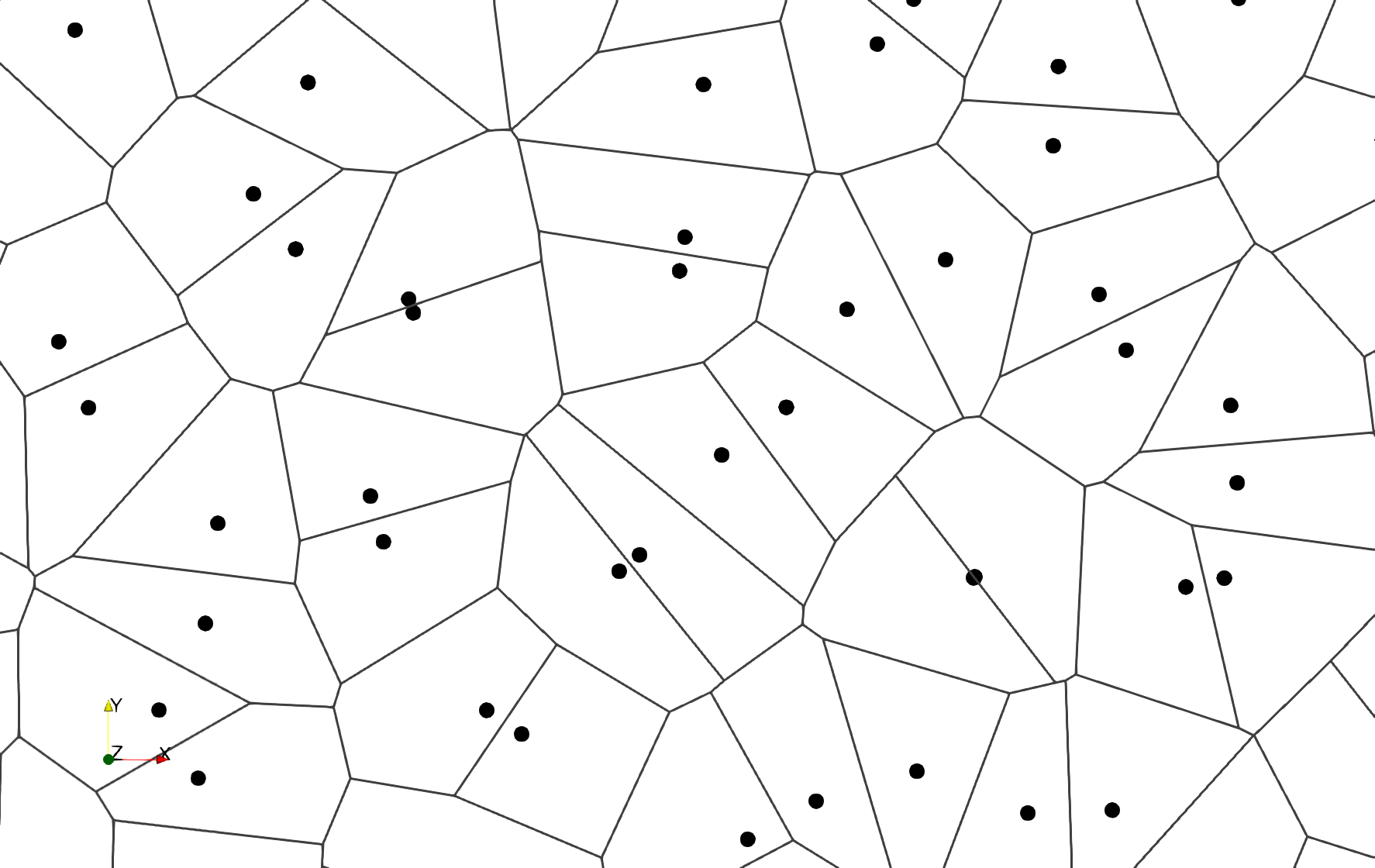}
		\caption{An example of low-quality Voronoi mesh. Some cells are strongly elongated and some pairs of generating seeds almost overlap.}
		\label{fig:badvoronoi}
	\end{figure}
	We are especially concerned about seed overlapping. Although bit-precise overlaps are extremely rare in fixed-point arithmetic, a very small separation $r_{ij} \ll \delta r$ can pose significant problems, especially since $r_{ij}$ appears frequently in the denominator of our scheme, see Equations \eqref{eq:sgrad} and \eqref{eq:adgrad}. It is necessary to prevent the seed from overlapping, either by adding an artificial repelling force \cite{despres2024lagrangian} or by moving the seeds with slightly different velocity \cite{springel2011hydrodynamic}, similarly to the particle shifting in SPH \cite{lind2012incompressible}. In this paper, the second approach is preferred, while ensuring the conservative nature of the method. The procedure involves a relaxations step, which updates the Voronoi seeds as follows:
	\begin{equation}
		\xx_i^{n+1} = \xx_i^{n} + \dt \bm{w}^n_i, 
	\end{equation}
	where $\bm{w}^n_i$ is a vector pointing to the centroid of the cell $\omega_i$, which we specify later. This relaxation procedure is motivated by Lloyd iterations and improves the quality of the Voronoi mesh by making it "nearly centroidal". We call $\bm{w}^n$ the \textit{repair velocity} and since it bears no physical meaning, the fields $\rho, \vv, e$ should be affected by it as little as possible. In terms of the extensive variables
	\begin{equation}
		\begin{pmatrix}
			M_i \\ 
			\bm{U}_i\\
			E_i
		\end{pmatrix}
		=
		\begin{pmatrix}
			|\omega_i| \rho_i \\
			|\omega_i| \rho_i \vv_i \\
			|\omega_i| \rho_i  e_i
		\end{pmatrix},
	\end{equation}
	and using a finite volume approach, this can be achieved by solving a discrete version of a linear convection problem with velocity $-\bm{w}^n$:
	\begin{align}
		\frac{M_i^{n+1} - M_i^n}{\dt} &= |\omega_i^n|\addiv{(\rho^n \bm{w}^n)}{i} + \mathcal{R}_i(\rho^n), \label{eq:remap_M} \\
		\frac{\bm{U}_i^{n+1} - \bm{U}^n}{\dt} &= |\omega_i^n| \addiv{(\rho^n \vv^n \otimes \bm{w}^n)}{i} + \mathcal{R}_i(\rho^n \vv^n), \label{eq:remap_U} \\
		\frac{E_i^{n+1} - E_i^n}{\dt} &= |\omega_i^n|\addiv{(\rho^n e^n \bm{w}^n)}{i} + \mathcal{R}_i(\rho^n e_i^n). \label{eq:remap_E}
	\end{align}
	Here, the derivatives are computed on cells $\omega^n_i$ (although using the updated mesh is also viable), and $\mathcal{R}$ is a Rusanov approximate Riemann solver
	\begin{equation}
		\mathcal{R}(\phi^n) = -\frac{1}{2}\sum_j  |\Gamma_{ij}|^n \max \{ |\bm{w}^n_i|, |\bm{w}^n_j| \} \, (\phi_i^n - \phi_j^n ) , \qquad \forall \phi^n,
	\end{equation}
	which adds some diffusion to prevent entropy from decreasing and avoids the formation of new extrema, especially in the density. It remains to discuss, how to select $\bm{w}_i^n$. We suggest the following definition:
	\begin{equation}
		\bm{w}_i^n = \frac{1}{\dt + \tau_i}(\bm{c}_i^n - \xx_i^n),
		\label{eq:w}
	\end{equation}
	where $\tau_i$ is a mesh relaxation parameter. To find a suitable formula for $\tau_i$, we postulate that $\tau$ must have units of time and take into account the quality of a cell. Both requirements can be met using 
	\begin{equation}
		\frac{1}{\tau_i} := \frac{\max_j |r_{ij}^n|^2}{\min_j |r_{ij}^n|^2} \|\mathbb{D}_i^n\|_\mathrm{Fro},
        \label{eq:relaxtime}
	\end{equation}
	where $\| \cdot \|_\mathrm{Fro}$ denotes the Frobenius norm and the extrema computed with respect to the set of all neighboring cells.
	
	
	
	Using Lloyd relaxation sacrifices the strict Lagrangian nature since the cell masses $M_i^{n+1}$ are no longer preserved. However, it is clear that $\bm{w} \to 0$ for $\delta r \to 0$ and so the method remains Lagrangian in an asymptotic sense.
	
	
	
	\subsection{Multi-phase projection} \label{sec:multiproj}
	A special care is required for multi-phase flows, where the mesh relaxation should not cause mixing of separate phases. In other words, we need to guarantee (in a local sense) that the flux of volume between different phases is zero. Mathematically, this motivates to introduce the operator $\mathcal{D}(\bm{w})_i$ defined by
	\begin{equation}
		\mathcal{D}(\bm{w})_i := \frac{1}{|\omega_i|}\sum_{j}\varphi_{ij} \frac{|\Gamma_{ij}|}{r_{ij}} \left( \bm{w}_{ij} \cdot (\bm{m}_{ij} - \oxx_{ij}) -  \overline{\bm{w}}_{ij} \cdot \xx_{ij} \right) = 0, \label{eq:multiphase_constraint}
	\end{equation}
	where the cross-phase indicator $\varphi$ is defined as follows:
	\begin{equation}
		\varphi_{ij} = \begin{cases}
			1 & \text{if $i$ and $j$ belong to different phase} \\
			0 & \text{otherwise}\\
		\end{cases}.
	\end{equation}
	As a corollary, the total area of every phase remains constant during the relaxation step (up to time discretization error). Since \eqref{eq:multiphase_constraint} is a linear constraint for $\bm{w}$, it can by enforced by defining a preliminary repair velocity field $\bm{w}^0$ in accordance to \eqref{eq:w} and then projecting it orthogonally to the linear space defined by \eqref{eq:multiphase_constraint} with respect to the scalar product
	\begin{equation}
		(f,g) = \sum_i |\Omega_i| f_i g_i.
	\end{equation}
	More concretely, we seek 
	$\bm{w} = \bm{w}^0 + \delta \bm{w}$,
	where $\bm{w}$ belongs to the kernel of $\mathcal{D}$ and $\delta \bm{w}$ is orthogonal to it. Then, $\bm{w}$ belongs to the range of the negative adjoint operators, hence 
	\begin{equation}
		\delta \bm{w}_i = -\mathcal{D}^*(Q)_i = -\frac{1}{|\omega_i|}\sum_{j}\varphi_{ij} \frac{|\Gamma_{ij}|}{r_{ij}} (\bm{m}_{ij} - \xx_i)Q_{ij},
	\end{equation}
	where $Q$ is unknown. We can determine $Q$ by solving the linear system
	\begin{equation}
		\mathcal{D}(\mathcal{D}^*(Q))_i =  \mathcal{D}(\bm{w}^0)_i.
	\end{equation}
	This system is similar to \eqref{eq:incompressible} and can be solved by a matrix-free conjugate gradient method. The linear operator is no longer a true composition of gradient and divergence, so the sparsification trick \eqref{eq:sparsification_trick} is no longer available. However, it is also not needed since the only non-zero terms are located on the material interface and the representing matrix is sufficiently sparse. The remapping \eqref{eq:remap_M}-\eqref{eq:remap_U} should be performed separately for every phase, so that mass, momentum and energy are conserved.
	
	\section{Boundary conditions} \label{sec:bcs}
	Naturally, the most simple to prescribe are \textit{periodic boundary conditions}, which merely require to construct a periodic Voronoi mesh. For the evaluation of the gradient operators, namely \eqref{eq:sgrad} and \eqref{eq:adgrad}, it is sufficient to take $\xx_j$ (appearing not only in $r_{ij}$) as the nearest point to $\xx_i$ within the given equivalence class defined by periodicity.
	
	When the Voronoi mesh is constrained to a convex boundary box and the fluid is inviscid, our equations will automatically impose that there exists no flux of mass, momentum and energy across the boundary. This corresponds to \textit{no-penetration condition} for an Euler fluid.
	
	The situation becomes more sophisticated with the inclusion of viscosity, where the tangential velocity component requires an additional constraint.	A \textit{free-slip condition} can be set by manipulating the expression for the velocity gradient as follows:
	\begin{equation}
		\begin{split}
			\mathbb{L}_i &= \adgrad{\vv}{i} = \frac{1}{|\omega_i|} \sum_j \frac{|\Gamma_{ij}|}{r_{ij}} \bigg( \vv_{ij} \otimes (\bm{m}_{ij} - \oxx_{ij}) -  \overline{\vv}_{ij} \otimes \xx_{ij} \bigg) \\
			&= \frac{1}{|\omega_i|} \sum_j \frac{|\Gamma_{ij}|}{r_{ij}} \bigg( \vv_{ij} \otimes (\bm{m}_{ij} - \oxx_{ij}) + \frac{1}{2}\vv_{ij} \otimes \xx_{ij} \bigg) \\
			&= \frac{1}{|\omega_i|} \sum_j \frac{|\Gamma_{ij}|}{r_{ij}} \bigg( \vv_{ij} \otimes (\bm{m}_{ij} - \xx_{j})\bigg).
		\end{split}
		\label{eq:freeslip}
	\end{equation}
	In general, this holds true for interior cells, which satisfy 
	\begin{equation}
		\sum_j \frac{|\Gamma_{ij}|}{r_{ij}} \vv_i \otimes \xx_{ij} = -\vv_i \otimes \int_{\partial \omega_i} \bm{n} = 0,
	\end{equation}
	but imposing it for boundaries leads to a free-slip condition by ensuring zero viscous stress whenever cells near the boundary move uniformly with a constant velocity.
	
	Frequently, we need to prescribe a \textit{Dirichlet boundary condition} for velocity:
	\begin{equation}
		\left. \bm{v} \right|_\Gamma = \bm{v}_D,
	\end{equation}
	where $\bm{v}_D \cdot \bm{n} = 0$ and $\Gamma$ is a piecewise linear part of the boundary. In such case, the implementation is same as in the free-slip condition, except that we also need to include the boundary friction force, yielding a modified irreversible step:
	\begin{align}
		\rho_i \frac{\tilde{\vv}_i^{n+1} - \vv_i^n}{\dt} &= \sdiv{\mathbb{S}^n}{i} + \bm{f}^\mathrm{bdary}_i \\
		\rho_i \frac{\tilde{e}_i^{n+1} - e_i^n}{\dt} &=  \sdiv{\mathbb{S}^n}{i} \cdot \vv^n_i + \mathbb{L}^n_i:\mathbb{S}^n_i + \bm{f}^\mathrm{bdary}_i \cdot \vv^n_i.
	\end{align}
	where the force $\bm{f}^\mathrm{bdary}_i$ is designed to enforce the boundary condition and can be prescribed using the technique of mirror cells:
	\begin{equation}
		\bm{f}^\mathrm{bdary}_i = \frac{-\frac{\mu_i}{|\omega_i|}\sum_{j \in \mathcal{M}_i} \frac{|\Gamma_{ij}|}{r_{ij}}(\vv^n_i - \vv^n_j) }{1 + \mu_i \dt \frac{1}{|\omega_i|}\sum_{j \in \mathcal{M}_i} \frac{|\Gamma_{ij}|}{r_{ij}}}, \label{eq:f_bdary}
	\end{equation}
	with $\mathcal{M}_i$ representing the set of (fictitious) mirror cells, which lie outside the domain $\Omega$. These are constructed by reflecting $\xx_i$ using the boundaries as the axes of symmetry. The velocities of these cells can be estimated by extrapolation: 
	\begin{equation}
		\vv_i' = 2\vv_D - \vv_i.
	\end{equation}
	In particular, the \textit{no-slip boundary condition} corresponds to $\vv_D = 0$. The denominator in \eqref{eq:f_bdary} is a limiter which guarantees that the wall friction cannot introduce new extrema in velocity, leading to a more stable boundary layer.
	
	Note that inflow and outflow conditions are difficult to describe in Lagrangian methods, because they necessitate either the generation or the elimination of computational nodes. We will not investigate this topic. Hopefully, it will be possible to use the SPH implementation of inflow condition \cite{lastiwka2009permeable} in the future.  
	
	Contrary to SPH, it is unclear how to generate a Voronoi mesh for a fluid with \textit{free surface}, for example a breaking water wave. Conceptually simplest way to simulate free surface is by adding an additional fluid phase with small density (air). However, it remains to be found how to treat multi-phase problems with large density ratios, which is a known challenge in water-air multi-material simulations.

	\section{Numerical results} \label{sec:results}
	
	\subsection{Gresho vortex}

	The Gresho vortex is a rare case of a two-dimensional benchmark for compressible (and incompressible) flows, where an analytical solution is available. It prescribes an initial velocity field
	\begin{equation}
		\vv = \begin{pmatrix*}[r]
			-\omega y\\
			\omega x
		\end{pmatrix*},
	\end{equation}
	where $\omega$, the angular velocity of the vortex, is given in terms of the radial coordinate $r=\sqrt{x^2+y^2}$ by
	\begin{equation}
		\omega(r) = \begin{cases}
			5, & r \leq \frac{1}{5} \\
			\frac{2}{r}-5, & r \in \left(\frac{1}{5}, \frac{2}{5}\right]\\
			0, & r > \frac{2}{5}
		\end{cases},
		\label{eq:gresho_v}
	\end{equation}
	the initial density is $\rho = 1$ and the initial pressure is given by
	\begin{equation}
		p(r) = \begin{cases}
			p_0 + \frac{25}{2}r^2, & r \leq \frac{1}{5} \\
			p_0 + \frac{25}{2}r^2 + 4(1 - 5r) + 4\ln 5r, & r \in \left(\frac{1}{5}, \frac{2}{5}\right]\\
			p_0 - 2 + 4\ln{2}. & r > \frac{2}{5}
		\end{cases}.
		\label{eq:gresho_P}
	\end{equation}
	Assuming an ideal gas, the value of $p_0$ can be related to the minimal sound speed by
	\begin{equation}
		p_0 = \frac{\rho_0 c_0^2}{\gamma},
	\end{equation}
	which in turn is linked to the Mach number as
	\begin{equation}
		\frac{1}{c_0} = \mathrm{Ma},
	\end{equation}
	since the characteristic velocity magnitude is 1. For very low Mach numbers, it is preferable to avoid catastrophic cancellation in pressure gradients using the \textit{stiffened gas equation of state} \cite{godunov1976numerical}, which allows to combine $p_0 = 0$ and an arbitrary value of $c_0$. In either case, we set $\gamma = 1.4$. The computational domain is $\Omega = (0,5)\times (0,5)$, that is discretized with $\delta r = 0.005$ and $\dt = 0.1\delta r$. We set the end time $t_\mathrm{end} = 3$, which corresponds to more than two full revolutions of the vortex. We prescribe free-slip conditions, although the boundary treatment is not important in this case.
	
	Equations \eqref{eq:gresho_v}-\eqref{eq:gresho_P} define a stationary solution of the inviscid Navier-Stokes equations. The solution is difficult to reproduce in a non-stationary solver, especially because the velocity field is non-differentiable. The results can be observed in Figures \ref{fig:gresho_vplot}-\ref{fig:gresho_midline}. We obtain very good results for Mach numbers ranging from 0.001 to 1. This test validates the ability of our numerical method to simulate low Mach number flows with a compressible solver, hence investigating the asymptotic preserving properties of the new scheme. It also shows that the artificial viscosity does not deteriorate the result of an inviscid simulation without shocks.
	
		\begin{figure}[!htb]
		\centering
		\includegraphics[width=0.4\linewidth]{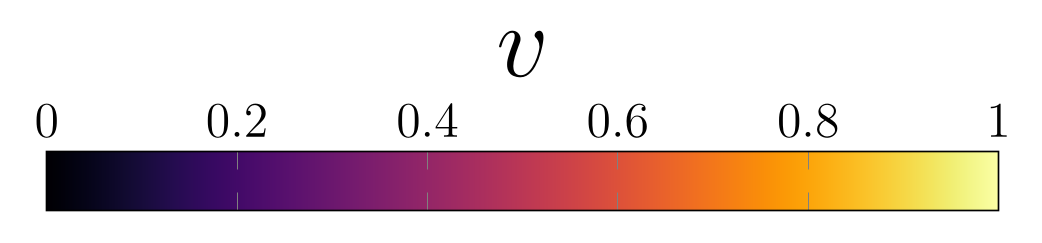}
		\begin{subfigure}{0.5\linewidth}
			\centering
			\includegraphics[width=\linewidth]{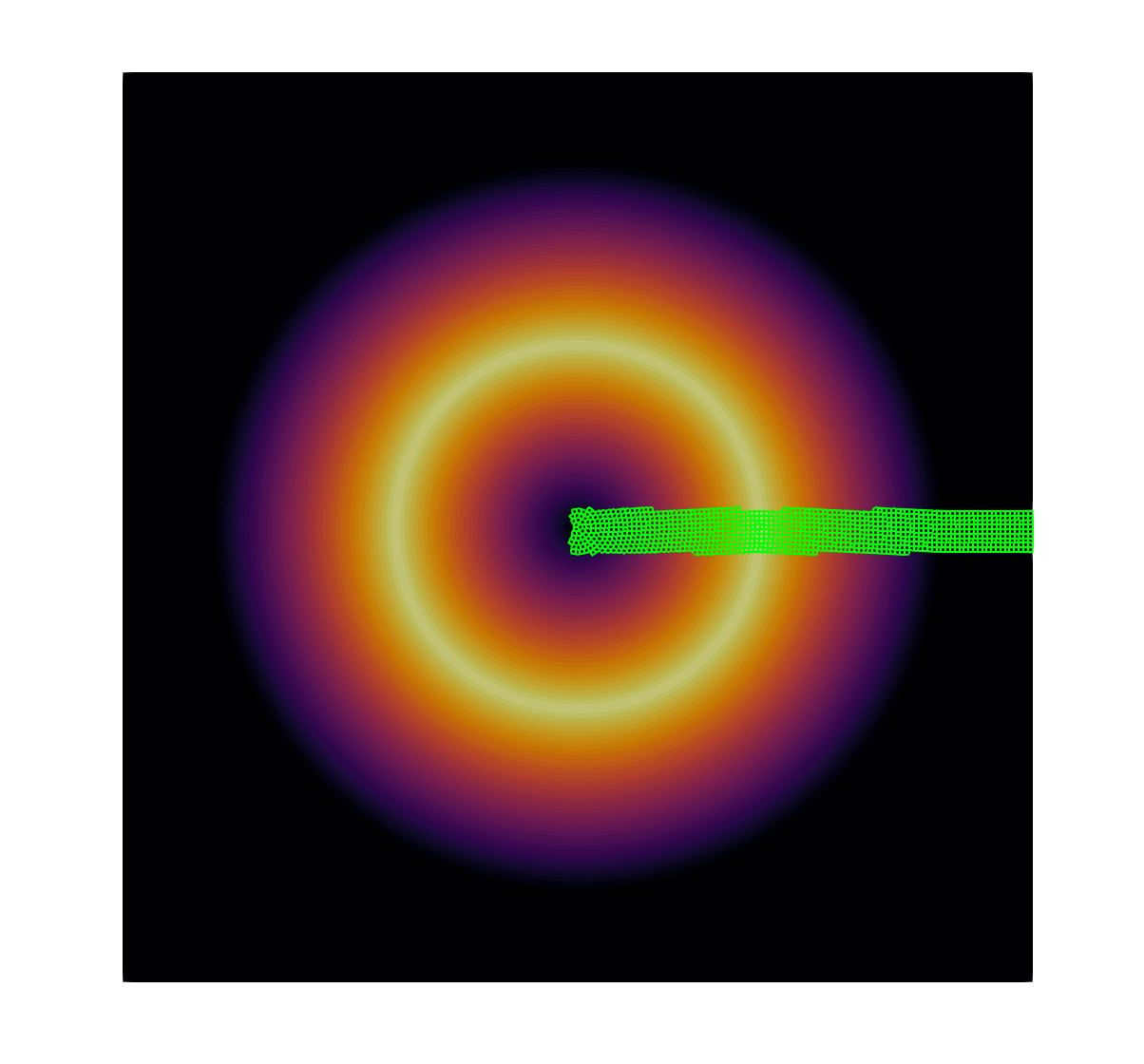}
			\caption{$t = 0$}
		\end{subfigure}%
		\begin{subfigure}{0.5\linewidth}
			\centering
			\includegraphics[width=\linewidth]{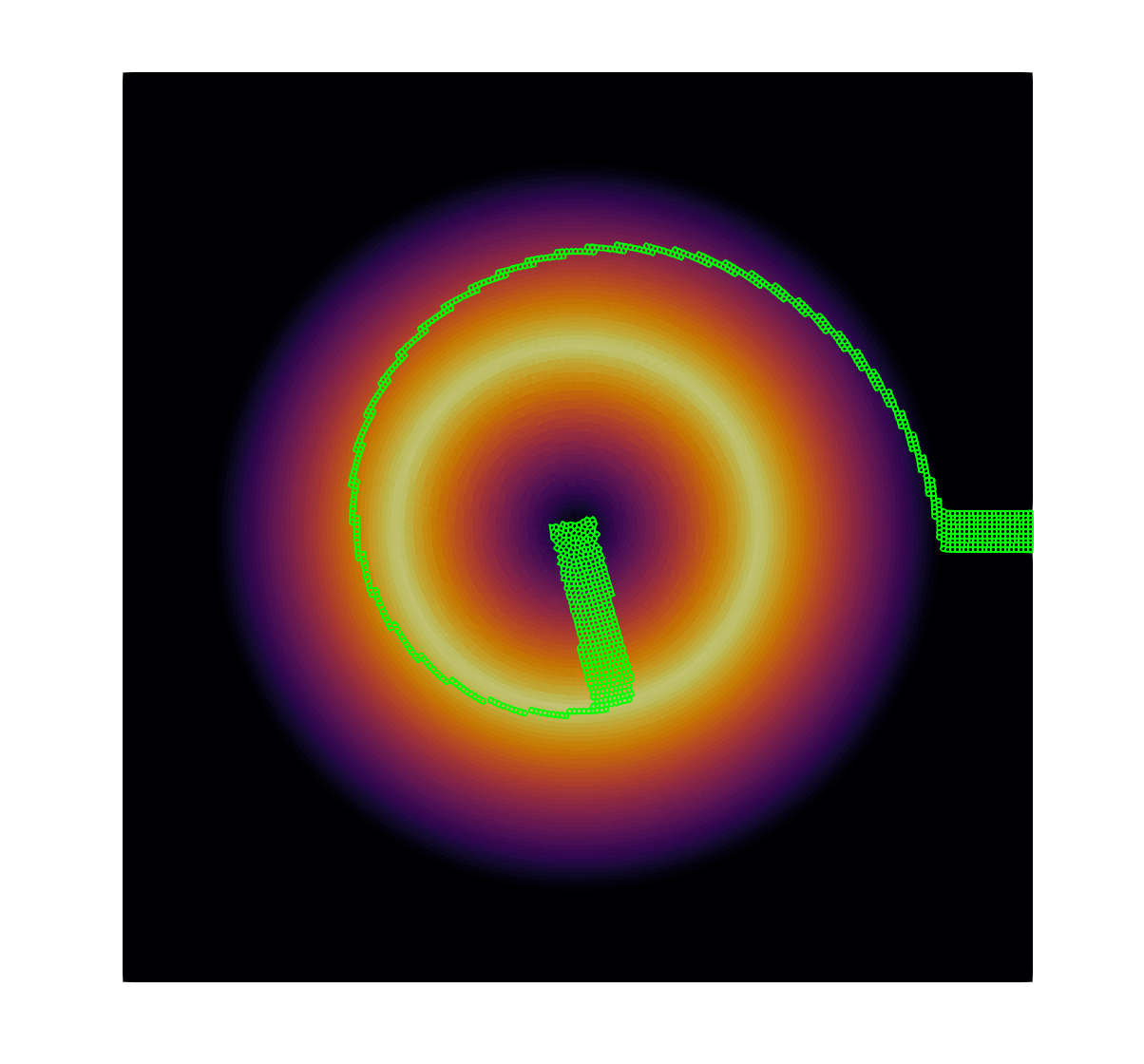}
			\caption{$t = 1$}
		\end{subfigure}
		
		\caption{The velocity color map of the Gresho vortex at $t=0$ and $t=1$. A group of quasi-Lagrangian cells is highlighted in green. The Voronoi grid consists of 200x200 cells.}
		\label{fig:gresho_vplot}
	\end{figure}
	
	\begin{figure}[!htb]
		\centering
		\includegraphics[width=0.75\linewidth]{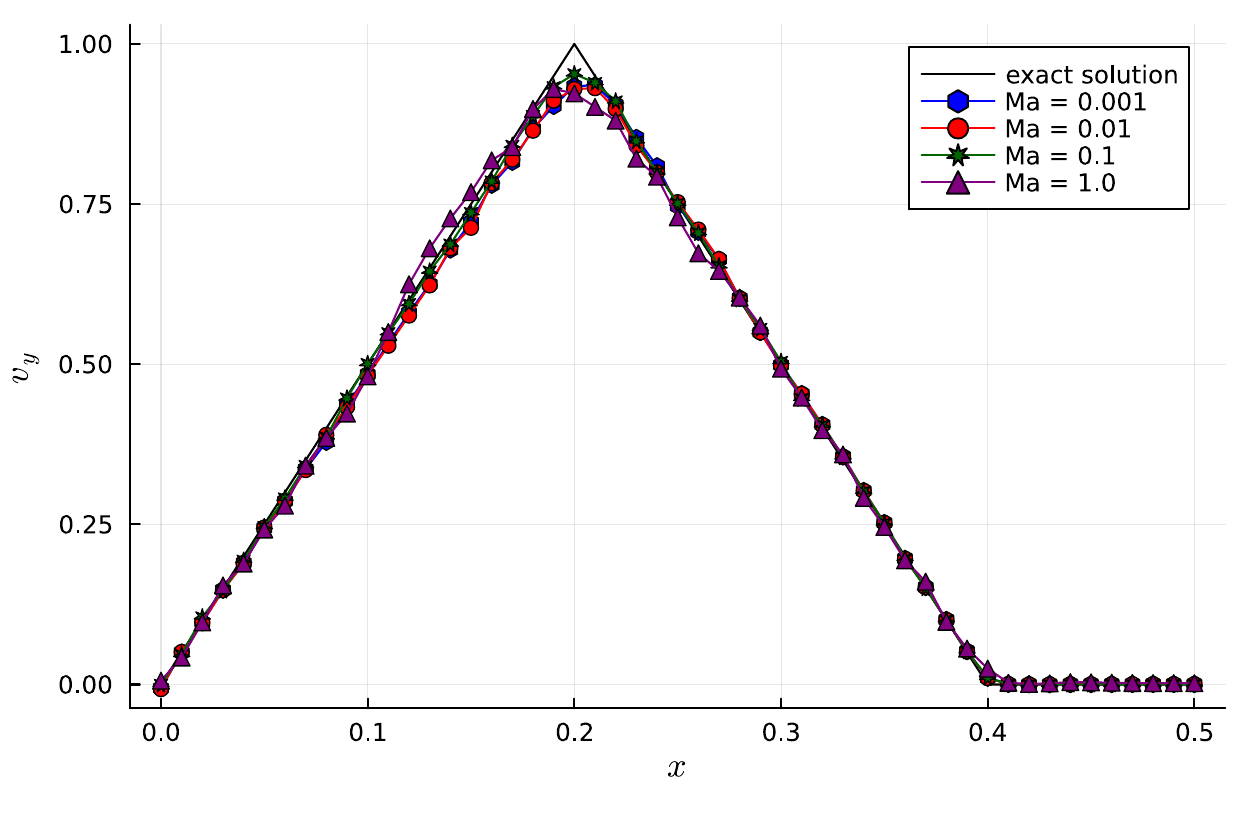}
		\caption{The $y$ component of the velocity along the positive $x$-axis at $t=3$. Results for various Mach numbers are shown. The error peaks at about $2\%$, which is an excellent agreement. For $\mathrm{Ma} = 0.001$, the stiffened gas equation of state was used.}
		\label{fig:gresho_midline}
	\end{figure}

	\subsection{Sedov problem}
	
	A classical two-dimensional shock problem is the Sedov blast wave
	\cite{sedov2018similarity}. Similarly to Riemann problems, it has a
	semi-analytical solution \cite{SedovExact}. The initial setup consists of an ideal fluid of
	low pressure $p_0$ and constant density $\rho_0 = 1$ in a square domain
	$\Omega = (-1.2,1.2)^2$. The boundary condition is unimportant, so for
	example the free-slip condition can be used. The analytical solution
	corresponds to the limit $p_0 \to 0$, but in our simulation, we use $p_0 =
	10^{-8}$. For this test, we use a hexagonal distribution of generating seeds
	with $\delta r = 0.01$. A small region of the same density and large
	pressure $p_1$ is defined in the center (the "ground zero") of the domain.
	Ideally, this region should be infinitesimally small, but for the
	simulation, we distribute the energy evenly between cells inside a small
	circle with radius $r = 5\delta r$. The pressure $p_1$ is related to the
	initial energy deposition
	\begin{equation}
		E_\mathrm{0} = \frac{\pi r^2 p_1}{\rho_0 (\gamma-1)}.
	\end{equation}
	Following \cite{Maire2009}, we set $\gamma = 1.4$ and $E_\mathrm{0} = 0.979264$ (we	multiplied the energy value by four since we are not restricting the
	simulation to the first quadrant). The solution forms a circular blast wave, with a density peaking at
	\begin{equation}
		\frac{\gamma + 1}{\gamma -1} = 6
	\end{equation}
	in the wake, and then diminishing to near vacuum. To faithfully capture the initial stage of the explosion, we use an adaptive time step
	\begin{equation}
		\dt = \frac{0.1 \delta r}{v_\mathrm{shock}},
	\end{equation}
	where
	\begin{equation}
		v_\mathrm{shock} = \sqrt{\frac{(\gamma +1)(\max_i p_i)}{2\rho_0}}
	\end{equation}
	is the estimated speed of the shock wave. The simulation stops at $t_\mathrm{end} = 1$. The result using 57 439 Voronoi cells is shown in Figures \ref{fig:sedov_points}-\ref{fig:sedov_graph}. The shape and the position of the density profile is in good agreement with the analytical solution, and a good preservation of the cylindrical symmetry can be observed. Some overshooting in density is observed, that might be cured by increasing the amount of artificial viscosity in \eqref{eq:artificial_visc}.

     \begin{figure}[!htbp]
        \centering
        \includegraphics[width=0.4\linewidth]{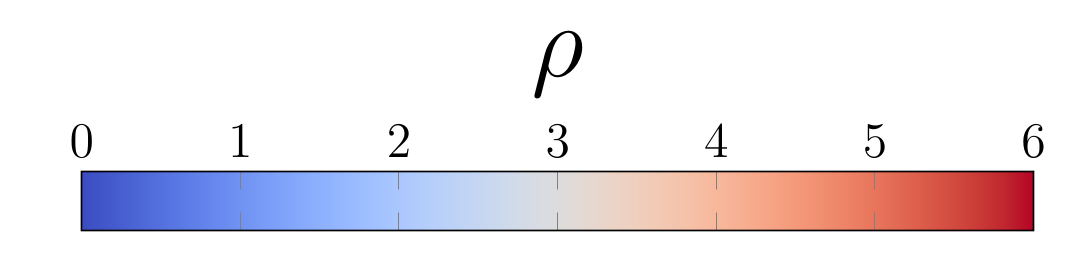}
        \begin{subfigure}{0.5\linewidth}
            \centering
            \includegraphics[width=0.8\linewidth]{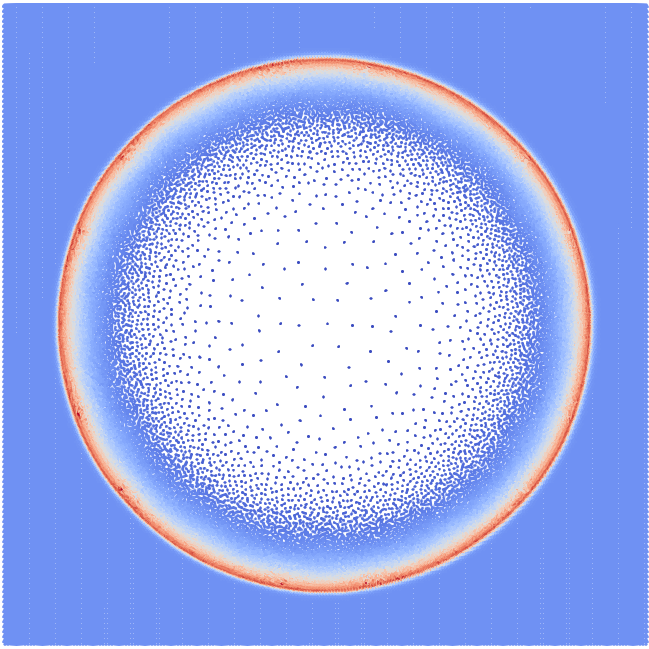}
            \subcaption{point cloud representation}
        \end{subfigure}%
		\begin{subfigure}{0.5\linewidth}
            \centering
            \includegraphics[width=0.8\linewidth]{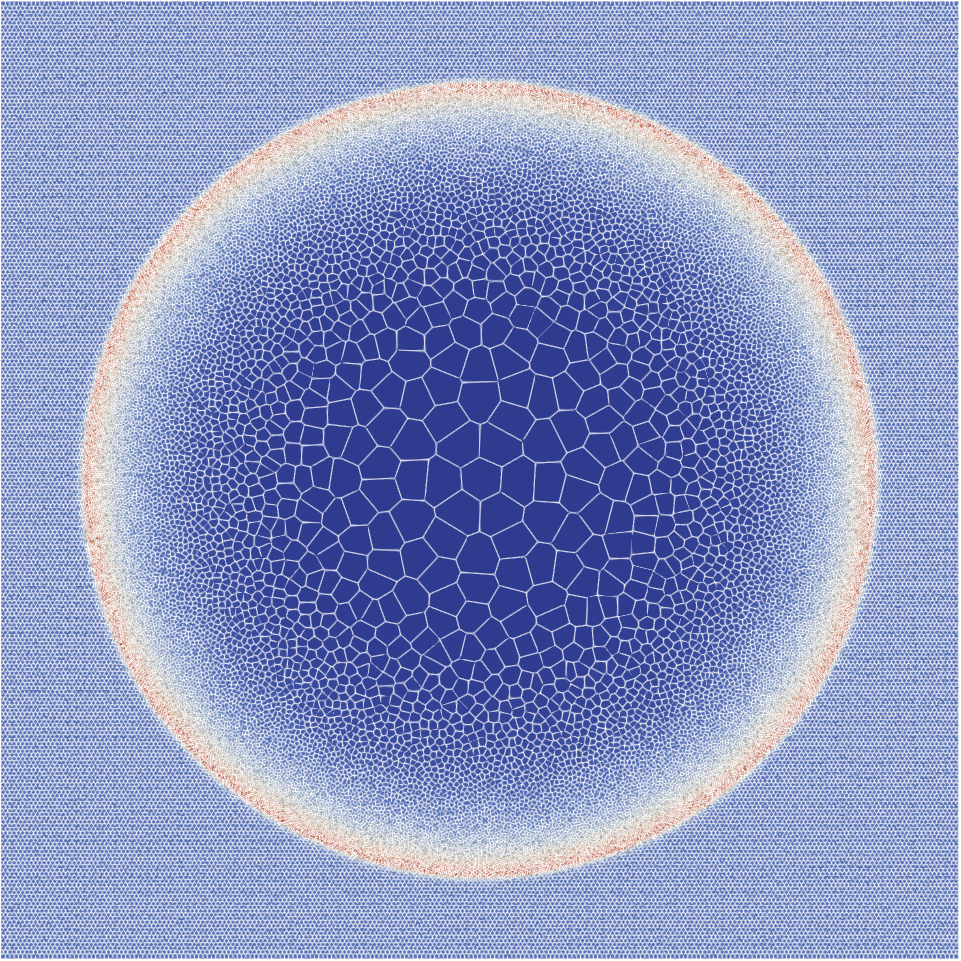}
            \subcaption{mesh representation}
        \end{subfigure}
		\caption{The color plot of density field in the Sedov test.}
        \label{fig:sedov_points}
    \end{figure}
    
	\begin{figure}[!htbp]
		\centering
		\includegraphics[width=0.7\linewidth]{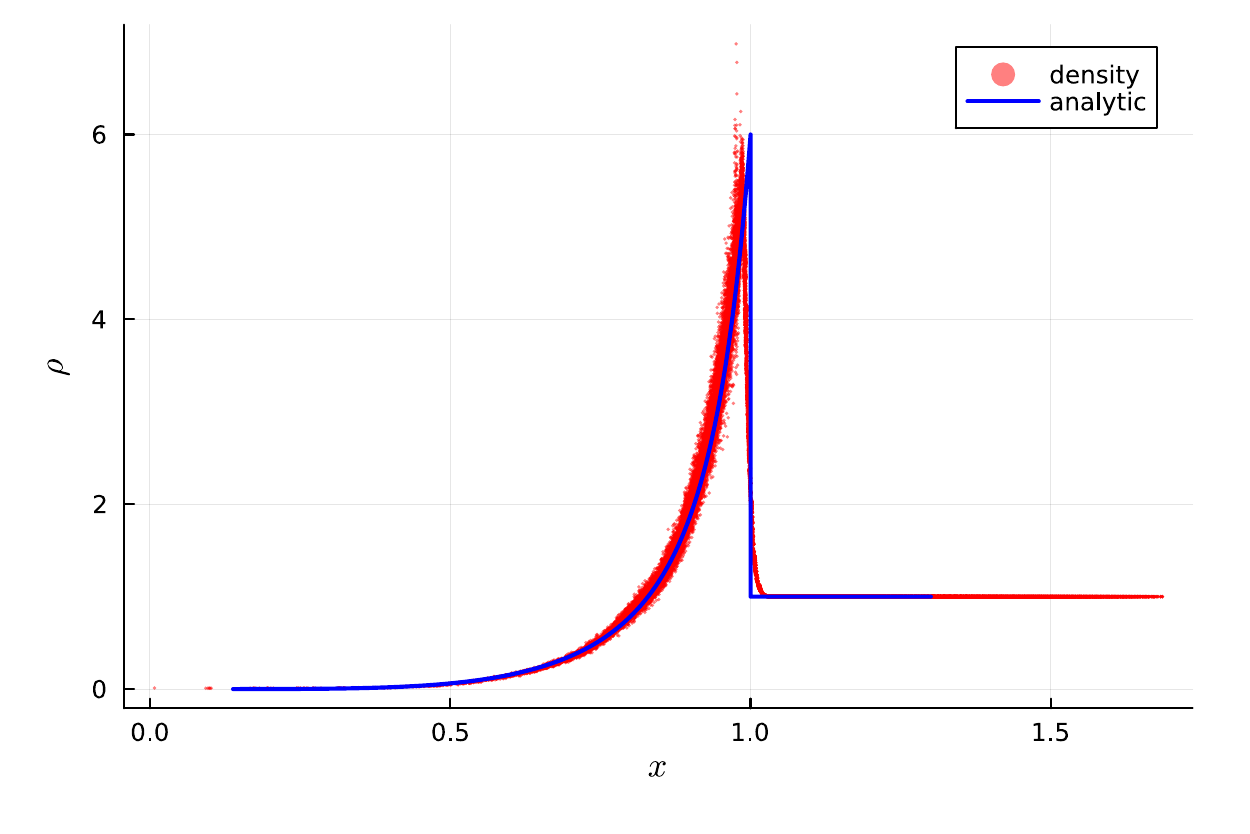}
		\caption{The density profile in Sedov test along the positive $x$-axis. The simulation result is red and the semi-analytical solution is blue.}
        \label{fig:sedov_graph}
	\end{figure}
	
	\subsection{Saltzman piston problem}

    \begin{figure}[!htbp]
        \centering
        \begin{subfigure}{0.75\linewidth}
            \centering
            \includegraphics[width=\linewidth]{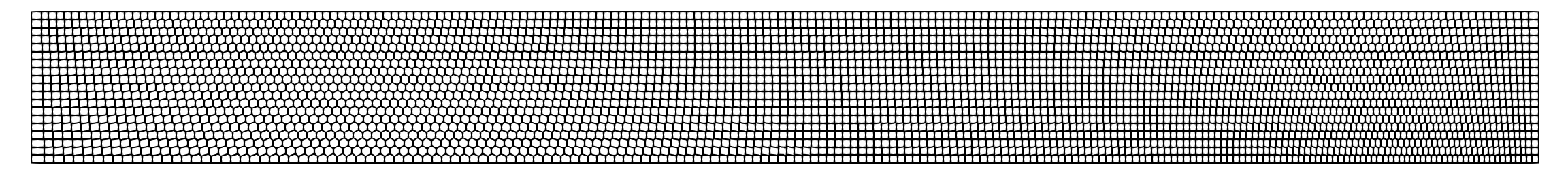}
            \subcaption{$t = 0$}
        \end{subfigure}
        \begin{subfigure}{0.75\linewidth}
            \centering
            \includegraphics[width=\linewidth]{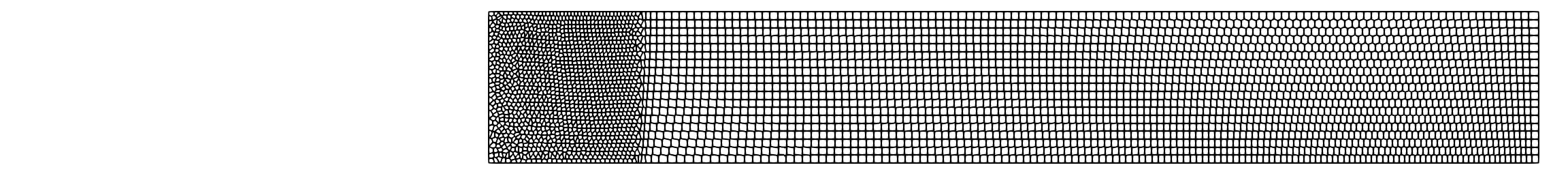}
            \subcaption{$t = 0.3$}
        \end{subfigure}
        \begin{subfigure}{0.75\linewidth}
            \centering
            \includegraphics[width=\linewidth]{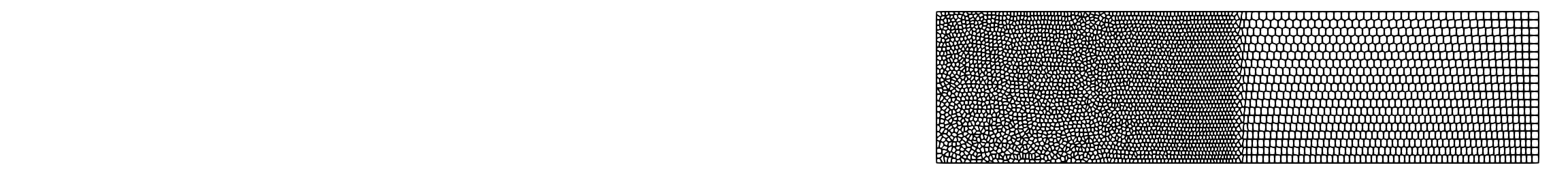}
            \subcaption{$t = 0.6$}
        \end{subfigure}%
        \caption{The Voronoi mesh in Saltzman problem at three different time frames.}
        \label{fig:piston_mesh}
    \end{figure}
    
    \begin{figure}[!htbp]
        \centering
        \begin{subfigure}{0.7\linewidth}
            \centering
            \includegraphics[width=\linewidth]{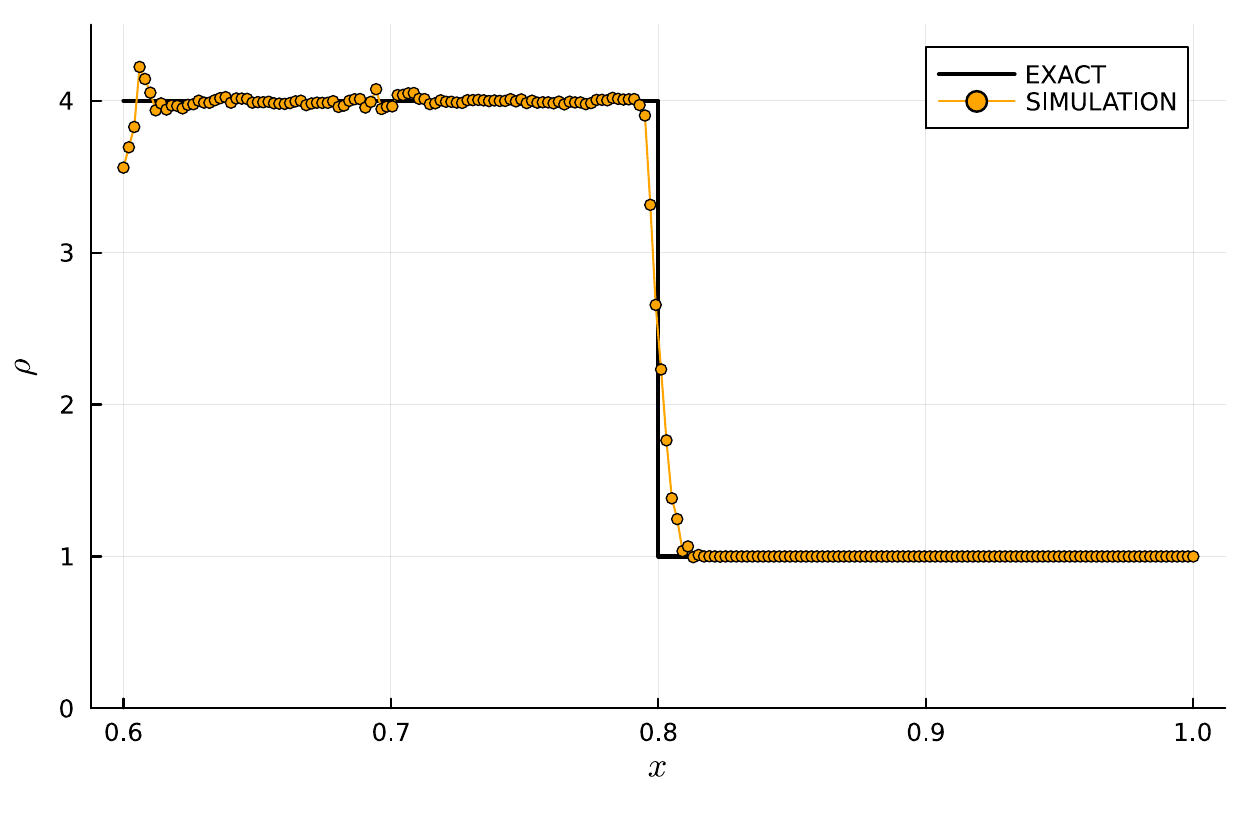}
            \subcaption{density}
        \end{subfigure}
        \begin{subfigure}{0.7\linewidth}
            \centering
            \includegraphics[width=\linewidth]{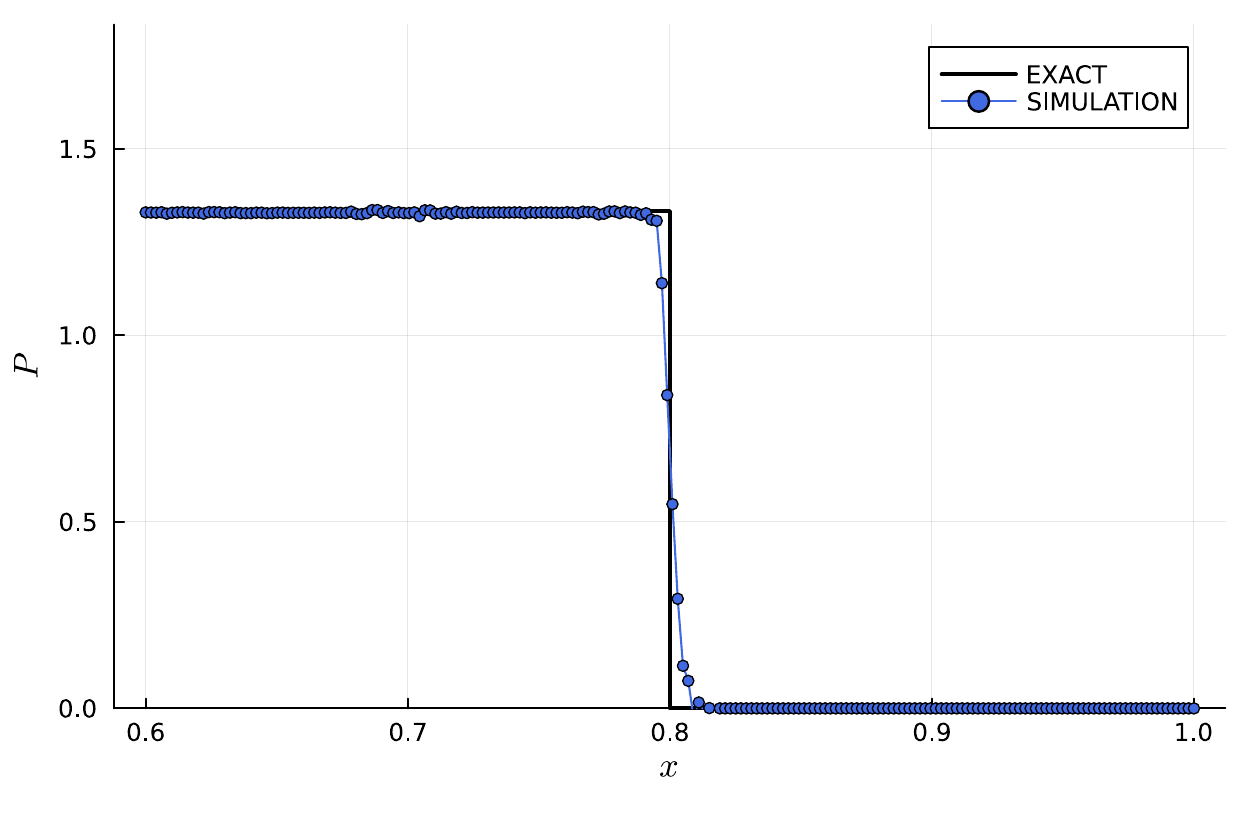}
            \subcaption{pressure}
        \end{subfigure}
        \caption{Comparison of the simulation result in Saltzman problem with the reference at $t = 0.6$.}
        \label{fig:piston_graph}
    \end{figure}

    Following \cite{dukowicz1992vorticity, campbell2001tensor, boscheri2017arbitrary}, let us consider a perfect gas confined in a time-dependent domain (a compressing piston) 
    \begin{equation}
        \Omega(t) = (t,1) \times (0, 0.1).
    \end{equation}
    The gas inside is initially at rest and has density $\rho_0 = 1$ and pressure $p_0$, the latter is a small constant parameter. The heat capacity ratio is $\gamma = \frac{5}{3}$. The fluid is inviscid and free-slip wall condition is prescribed on every side of the domain. That means:
    \begin{equation}
        \vv|_{\partial \Omega} \cdot \bm{n} = \begin{cases}
            -1 & \text{at the left,} \\
            \; 0 & \text{at the right, up and bottom.}
        \end{cases}
    \end{equation}
    The incompatibility between initial and boundary data leads to the emergence of a shock discontinuity, which propagates from left to right at speed
    $v_\mathrm{shock} = \frac{4}{3}$. In the strong shock limit ($p_0 \to 0$), the flow field is solvable analytically using Rankine-Hugoniot conditions, and reads:
    \begin{equation}
        (\rho, v_x, v_y, p) = \begin{cases}
            (4, 1, 0, \frac{4}{3}), & x < \frac{4t}{3}, \\
            (1, 0, 0, p_0), & x > \frac{4t}{3}.
        \end{cases}
        \label{eq:piston_ref}
    \end{equation}
    For the numerical treatment, we select $p_0 = 10^{-4}$. By default, in the literature, the computation is  performed on a uniform Cartesian mesh skewed by the mapping
    \begin{equation}\label{eq:piston_mapping}
        \begin{split}
            x' &= x + (0.1 - y)\sin 2\pi x,\\
            y' &= y,
        \end{split}
    \end{equation}
    so that the shock is no longer perfectly aligned with the mesh, increasing the difficulty of this test problem. In our case, transformation \eqref{eq:piston_mapping} is applied to the set of generating seeds at $t = 0$. The spatial resolution is $\delta r = \frac{1}{200}$ and the time step is $\delta t = 0.1\delta r$. The simulation is terminated at $t_\mathrm{end} = 0.6$.

    The time dependent domain requires a careful implementation and several aspects of the scheme need to be modified. First and foremost, the right hand side term $b_i$ in \eqref{eq:linprob3} must be updated to take into account the change of $|\omega_i|$ (thus density, thus pressure) from moving boundary. On semi-discrete level, the time derivative of a cell area becomes
    \begin{equation}
        \dv{|\omega_i|}{t} = |\omega_i| \addiv{\vv}{i} + \int_{\partial \omega_i \cap \partial \omega} \vv|_{\partial \Omega} \cdot \bm{n} \; \dd S =: |\omega_i| \langle \divg^\dag \vv \rangle.
    \end{equation}
    Correspondingly, we are solving \eqref{eq:linprob3} where
    \begin{equation}
        b_i = |\omega_i|^{n+1}\left( \frac{p^{n}_i }{\rho_i^{n+1}(c^{n}_i \dt)^2} -\frac{1}{\dt} \langle \divg^\dag \vv \rangle \right).
    \end{equation}
    It must be ensured that the generating seeds cannot escape from the fluid domain, which would lead to undefined behavior. To this end, the Voronoi cells on the left boundary are accelerated by the wall force
    every time step before the positional update so that $v_{\xx_i} = 1$,
    incurring the change of energy $e_i$ through the kinetic term along the way.
    Once the positions $\xx_i$ are updated, the domain boundaries are redefined
    as $\Omega = \Omega(t)$. In spite of $v_{\xx_i} = 1$ being enforced
    strongly, the cells may become cropped by the moving piston (if they were
    not identified as boundary cells in the previous time step and moved with
    slower horizontal velocity $v_{\xx_i} < 1$). Their density will be updated
    automatically by virtue of the relation \eqref{eq:bomass_implicit} but their
    energy needs to be updated by hand assuming their adiabatic compression.
    Elsewhere, the numerical scheme is unaltered. In particular, we did not tune any parameters of artificial viscosity \eqref{eq:artificial_visc} or
    mesh relaxation \eqref{eq:relaxtime}.

    The results are depicted in Figures \ref{fig:piston_mesh}-\ref{fig:piston_graph}, indicating a solid agreement with the strong-shock solution \eqref{eq:piston_ref} up to a certain density fluctuation near the (ever problematic) left boundary.

    \subsection{Triple point}
 
	Another interesting compressible test is the the triple point benchmark \cite{loubere2010reale}, which deals with a three-phase configuration. The domain is a rectangle $\Omega = (0,7)\times(0,3)$ with a free-slip condition on all sides. The domain is split into three sub-domains, namely $\Omega_1, \Omega_2, \Omega_3$, each containing a different ideal gas. The geometrical parameters and the values of initial density, pressure and adiabatic indices are
	\begin{equation}
		\begin{matrix*}[l] 
			\Omega_1 = (0,1)\times (0,3), & \rho_1 = 1, & p_1 = 1, & \gamma_1 = 1.5,\\
			\Omega_2 = (1,7)\times \left(0,\frac{3}{2}\right), & \rho_2 = 1, & p_2 = 0.1, & \gamma_2 = 1.4,\\
			\Omega_3 = (1,7)\times \left(\frac{3}{2},3\right), & \rho_3 = 0.125, & p_3 = 0.1, & \gamma_3 = 1.5.
		\end{matrix*}
	\end{equation}
	This setup develops a shock traveling from left to right with a spiral pattern at the triple point. The viscosity is zero, and the resolution used in our simulation is $\delta r = 0.01$. The result at final time $t=3$ is depicted in Figure \ref{fig:triplepoint} and it is qualitatively in good agreement with a finite volume result of the Baer-Nunziato model \cite{ferrari2024numerical}. This example highlights the importance of the multi-phase projection technique described in Section \ref{sec:multiproj}. When turned off, we observe an artificial ridge of density field in the wake of the shock between the first and the third fluid components (Figure \eqref{fig:tp_comparison}). 
	\begin{figure}[!htbp]
		\centering
		\includegraphics[width=0.4\linewidth]{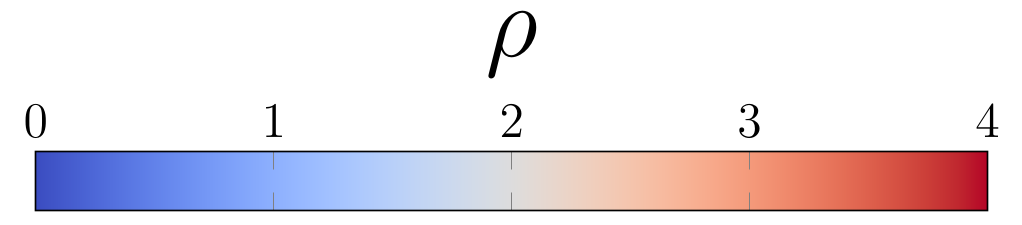}
		\begin{subfigure}{0.5\linewidth}
			\centering
			\vspace{5mm}
			\includegraphics[width=0.8\linewidth]{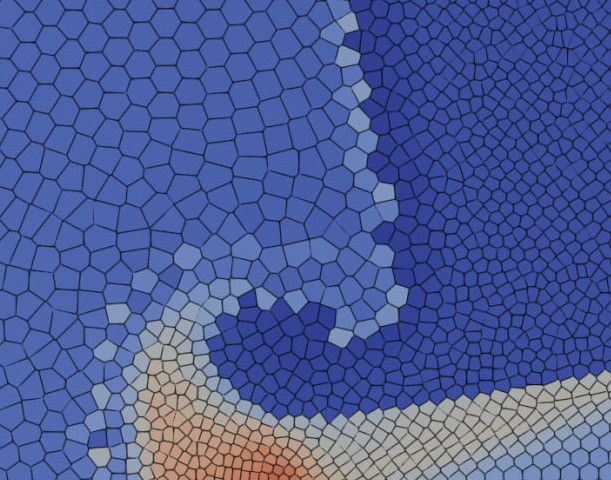}
			\subcaption{naive approach}
		\end{subfigure}%
		\begin{subfigure}{0.5\linewidth}
			\centering
			\includegraphics[width=0.8\linewidth]{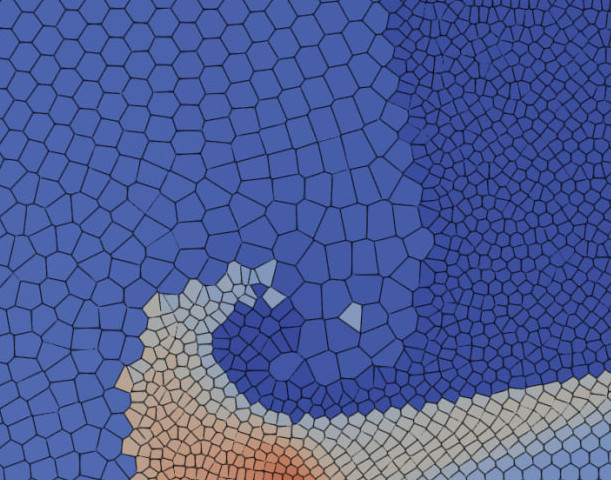}
			\subcaption{with the multi-phase projection}
		\end{subfigure}
		\caption{A detail of the spiral pattern at $t=1$ in Triple point problem  at low resolution.}
		\label{fig:tp_comparison}
	\end{figure}

\begin{figure}[!htbp]
	\centering
	\begin{subfigure}{\linewidth}
		\centering
		\includegraphics[width=0.7\linewidth]{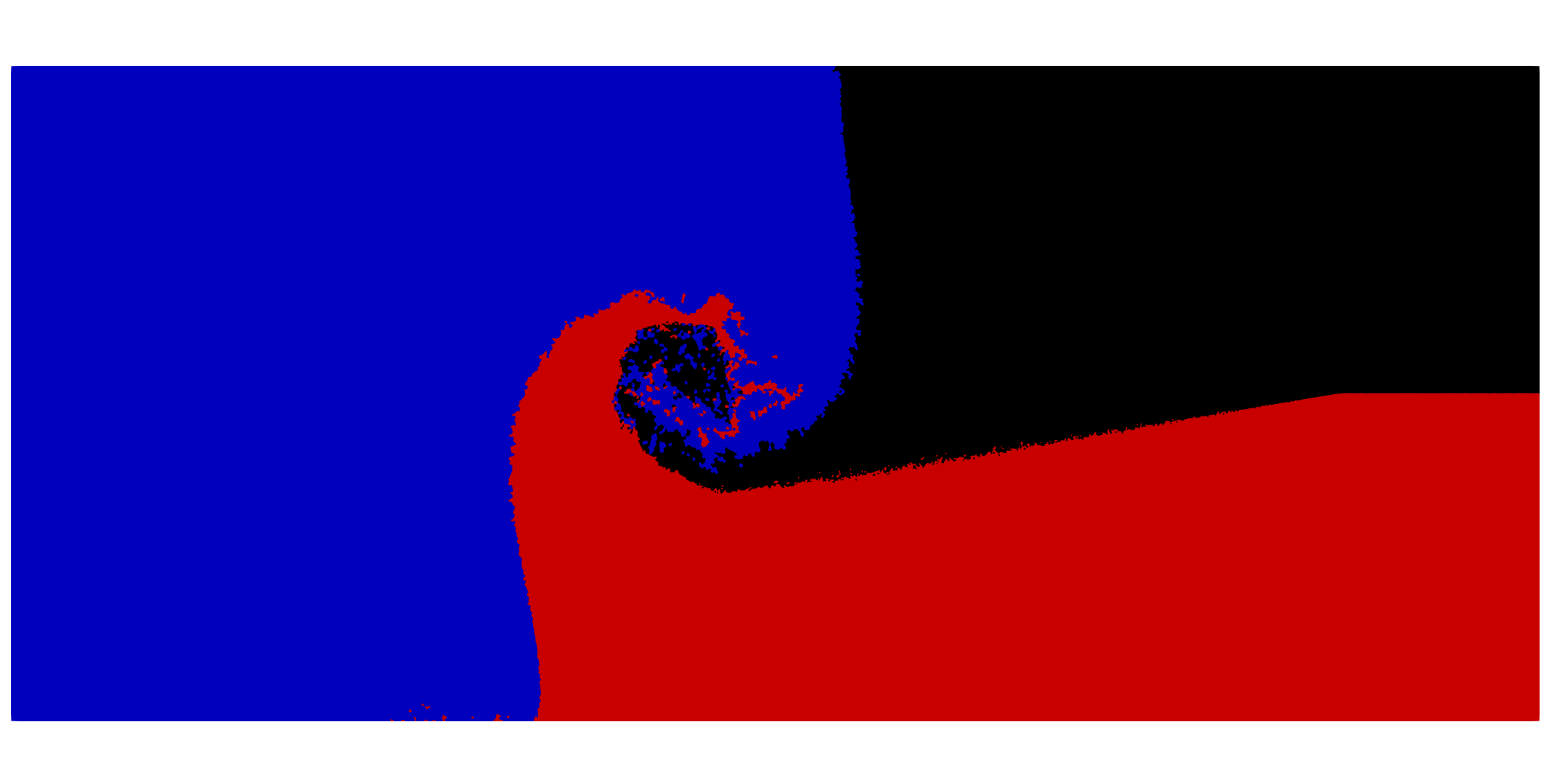}
		\caption{The three phases in blue, red and black respectively.}
	\end{subfigure}
	\newline
	\begin{subfigure}{\linewidth}
		\centering
		\includegraphics[width=0.4\linewidth]{images/coolwarm_colormap04.png}
		\includegraphics[width=0.7\linewidth]{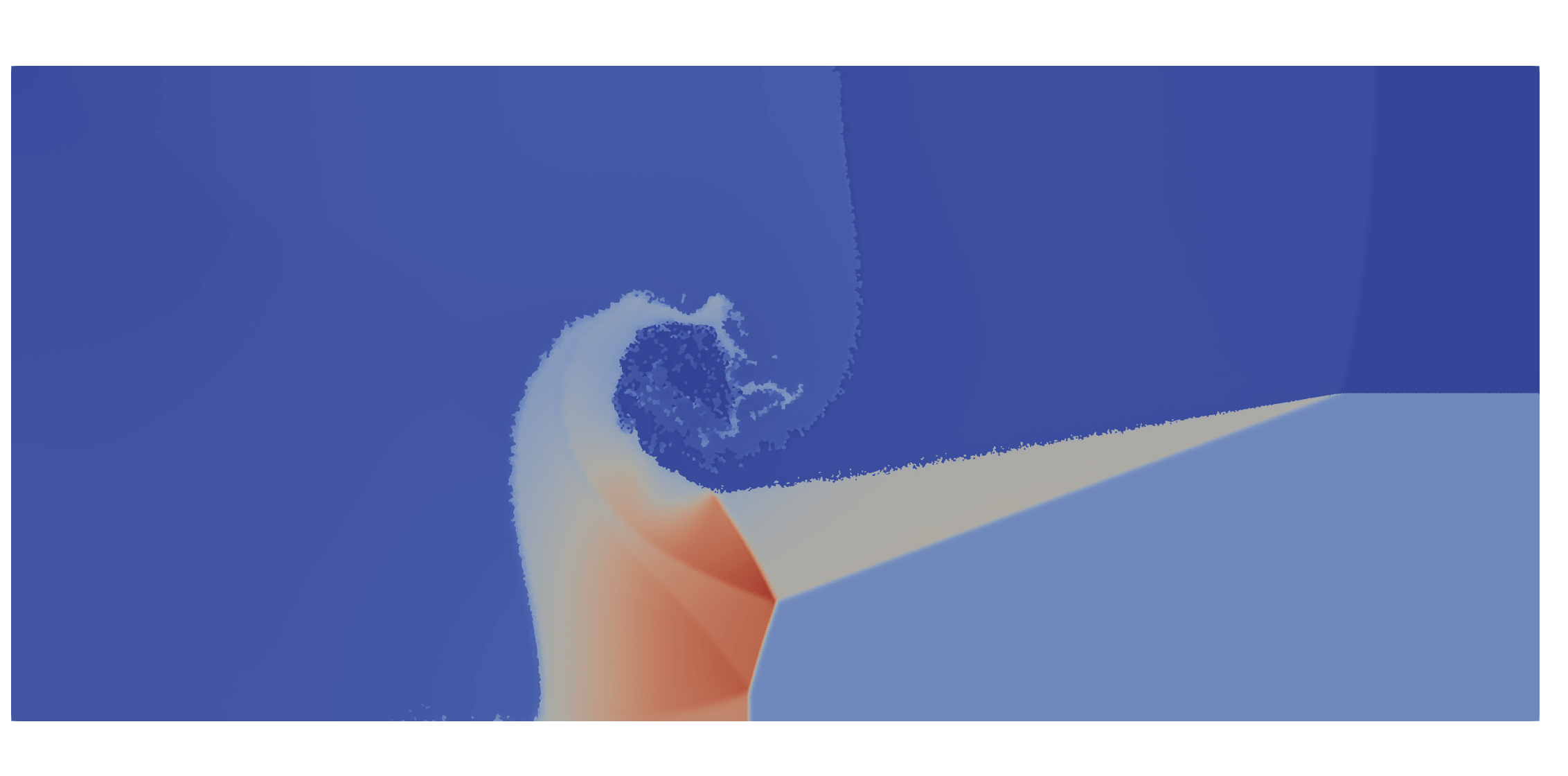}
		\caption{Density field.}
	\end{subfigure}
	\caption{The triple point benchmark at $t = 3$.}
	\label{fig:triplepoint}
\end{figure}
	
	\subsection{Double shear layer}
	The next test includes physical viscosity for the first time in this section. An ideal fluid with density $\rho = 1$, $\gamma = 1.4$ and pressure $p = \frac{100}{\gamma}$ is confined in a square $\Omega = (0,1)^2$ with periodic boundary condition. The initial velocity field has Cartesian components:
	\begin{align}
		v_x &= \begin{cases}
			\tanh\left(\xi \left(y - \frac{1}{4}\right) \right), &y < \frac{1}{2}\\
			\tanh\left(\xi \left(\frac{3}{4} - y\right) \right), &y \geq \frac{1}{2}
		\end{cases}, \\
		v_y &= \delta\sin(2 \pi x),
	\end{align}
	where $\delta = 0.05$, $\xi = 30$ and the dynamic viscosity coefficient is $\mu = 2\cdot 10^{-4}$. The mesh resolution is 200x200 and the termination time is $t = 1.6$. This setup generates a Kelvin-Helmholtz instability depicted in Figure \ref{fig:double_shear}. The result is in qualitative agreement with \cite{dumbser2016high}. To compute the $z$ component of vorticity, we use
	\begin{equation}
		\hat{\bm{z}} \cdot \rot \vv(\xx_i) \approx \left \langle \divg \begin{pmatrix}
			v_y \\
			-v_x
		\end{pmatrix} \right\rangle_{i}.
	\end{equation}
	
	\begin{figure}[!htbp]
		\centering
		\includegraphics[width=0.4\linewidth]{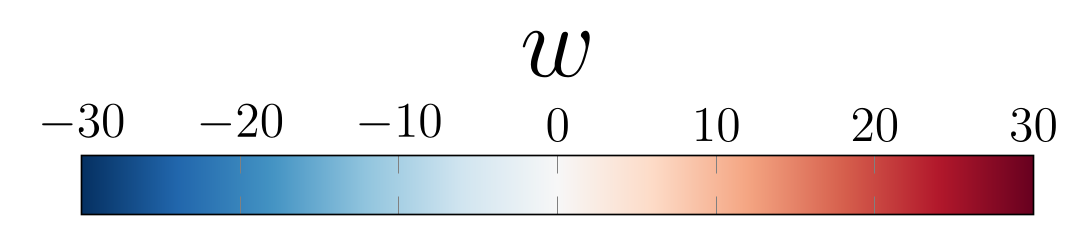}
		\begin{subfigure}{0.5\linewidth}
			\centering
			\includegraphics[width=\linewidth]{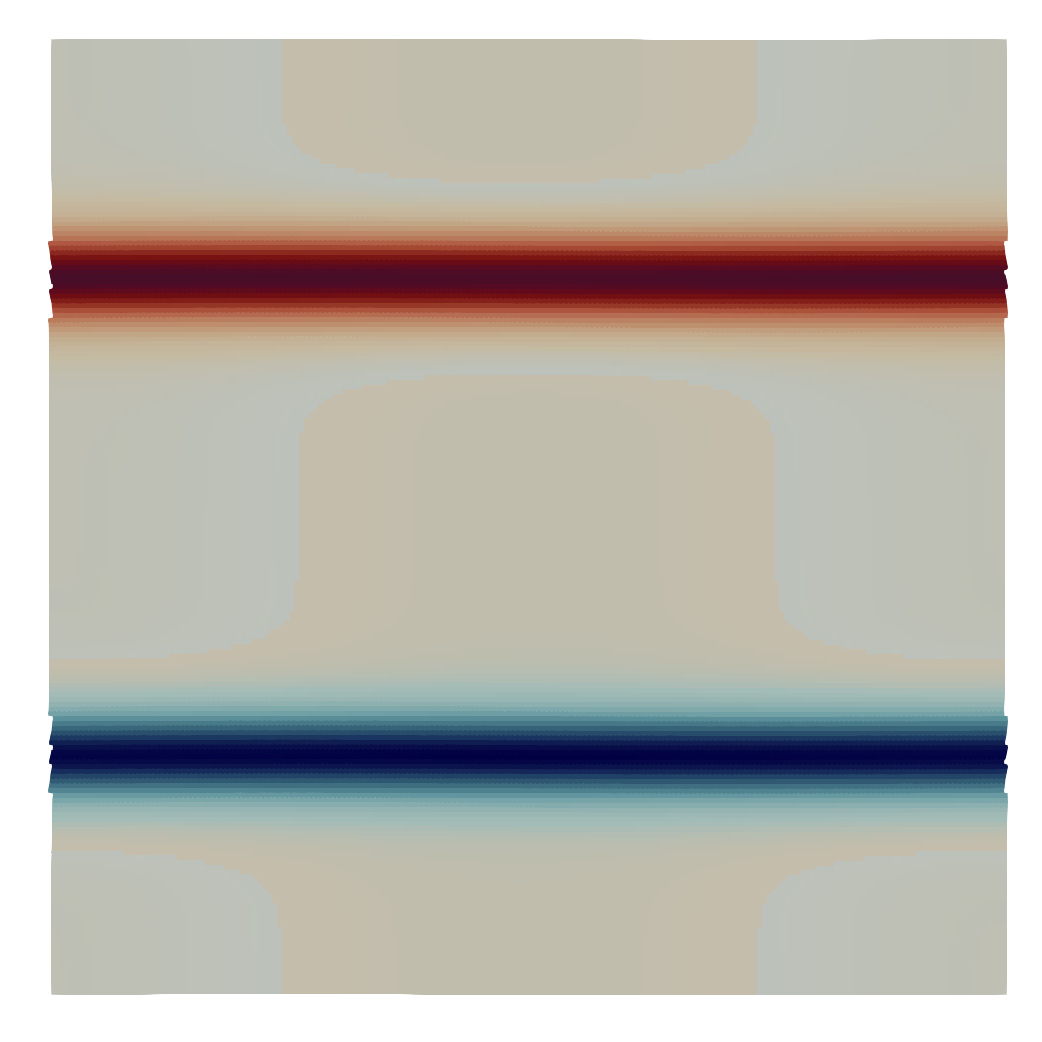}
			\caption{$t = 0$}
		\end{subfigure}%
		\begin{subfigure}{0.5\linewidth}
			\centering
			\includegraphics[width=\linewidth]{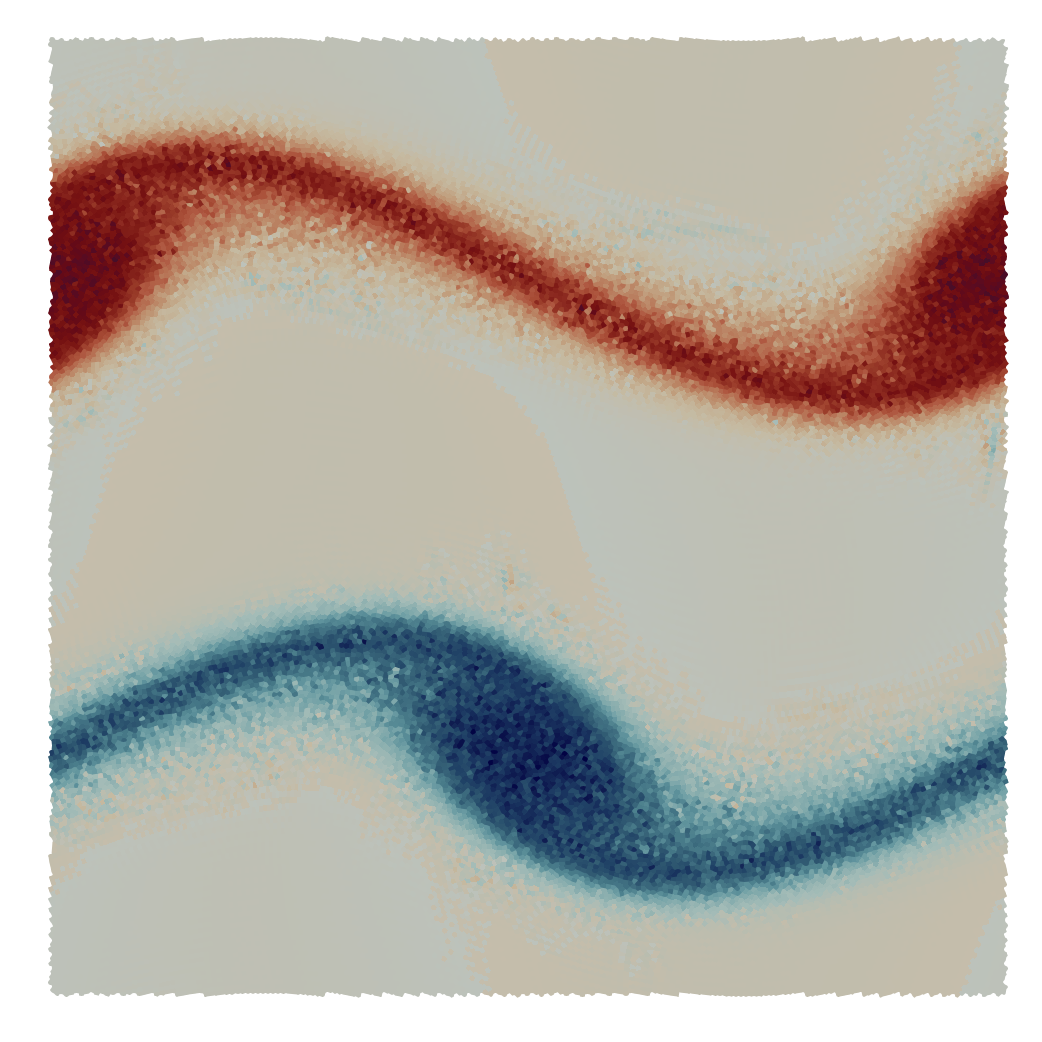}
			\caption{$t = 0.8$}
		\end{subfigure}
		\begin{subfigure}{0.5\linewidth}
			\centering
			\includegraphics[width=\linewidth]{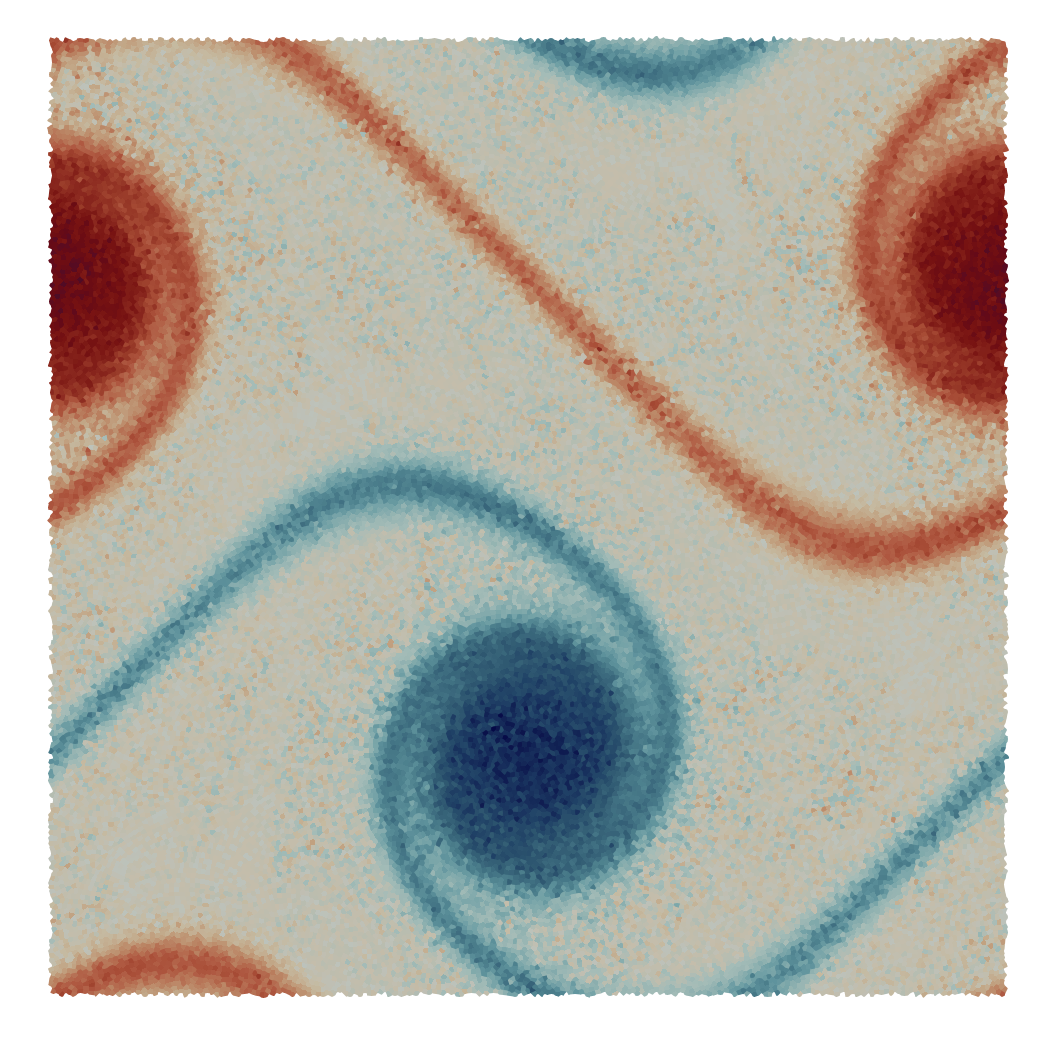}
			\caption{$t = 1.2$}
		\end{subfigure}%
		\begin{subfigure}{0.5\linewidth}
			\centering
			\includegraphics[width=\linewidth]{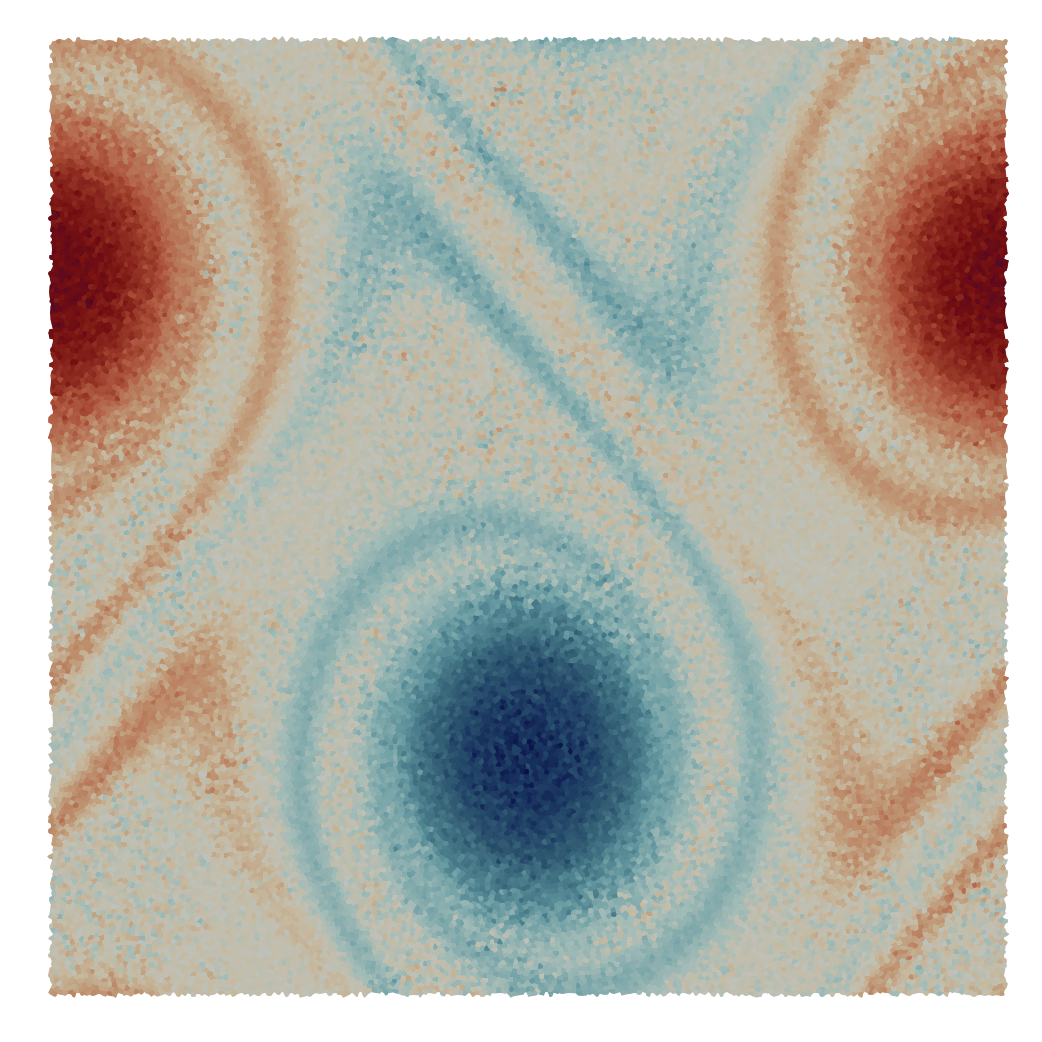}
			\caption{$t = 1.8$}
		\end{subfigure}
		\caption{The plot of vorticity $w$ in the double shear layer benchmark.}
		\label{fig:double_shear}
	\end{figure}
	
	\subsection{Taylor-Green vortex} \label{sec:tagr}

	A two-dimensional Taylor-Green vortex is a simple benchmark, which allows us to assess the convergence rate of SILVA. The domain in this problem is a square $\Omega = (0,1)^2$ with periodic boundary conditions on all sides. The initial condition and the solution are analytically described by
	\begin{equation}
		\vv(t, x, y) = V(t) \begin{pmatrix*}[r]
			\cos (2\pi x) \; \sin ( 2 \pi y) \\
			-\sin (2\pi x) \; \cos (2\pi y)
		\end{pmatrix*},
		\label{eq:tagr_v}
	\end{equation}
	and
	\begin{equation}
		p(t, x, y) = \frac{1}{2}V(t)^2 \bigg(\sin^2 (2 \pi x) + \sin^2(2 \pi y) - 1\bigg), 
		\label{eq:tagr_p}
	\end{equation}
	where the peak velocity $V$ decays exponentially as
	\begin{equation}
		V(t) = \exp\left( -\frac{4\pi^2t}{\mathrm{Re}}\right).
	\end{equation}
	Therefore, the velocity and pressure fields are solution to the dimensionless incompressible Navier-Stokes equations with $\rho = 1$ and Reynolds number $\mathrm{Re}$. By nature of the problem, the pressure is determined up to an additive constant. Formula \eqref{eq:tagr_p} is here normalized such that the average value is zero. To mimic an incompressible behavior with our compressible solver, it is suitable to use a stiffened gas equation of state with a large sound speed of $c = 1000$. The final time of the simulation is $t_\mathrm{end} = 0.2$. The resulting convergence curves are depicted in Figure \ref{fig:tagr_convergence}, and the numerical values of $L^2$ error in Table \ref{tab:tagr}. The experimental order of convergence (EOC) is about 1. So far, we did not have the ambition of designing a second order scheme and we obtained the expected result.  
	
	\begin{figure}[!htbp]
		\centering
		\begin{subfigure}{0.5\linewidth}
			\includegraphics[width=\linewidth]{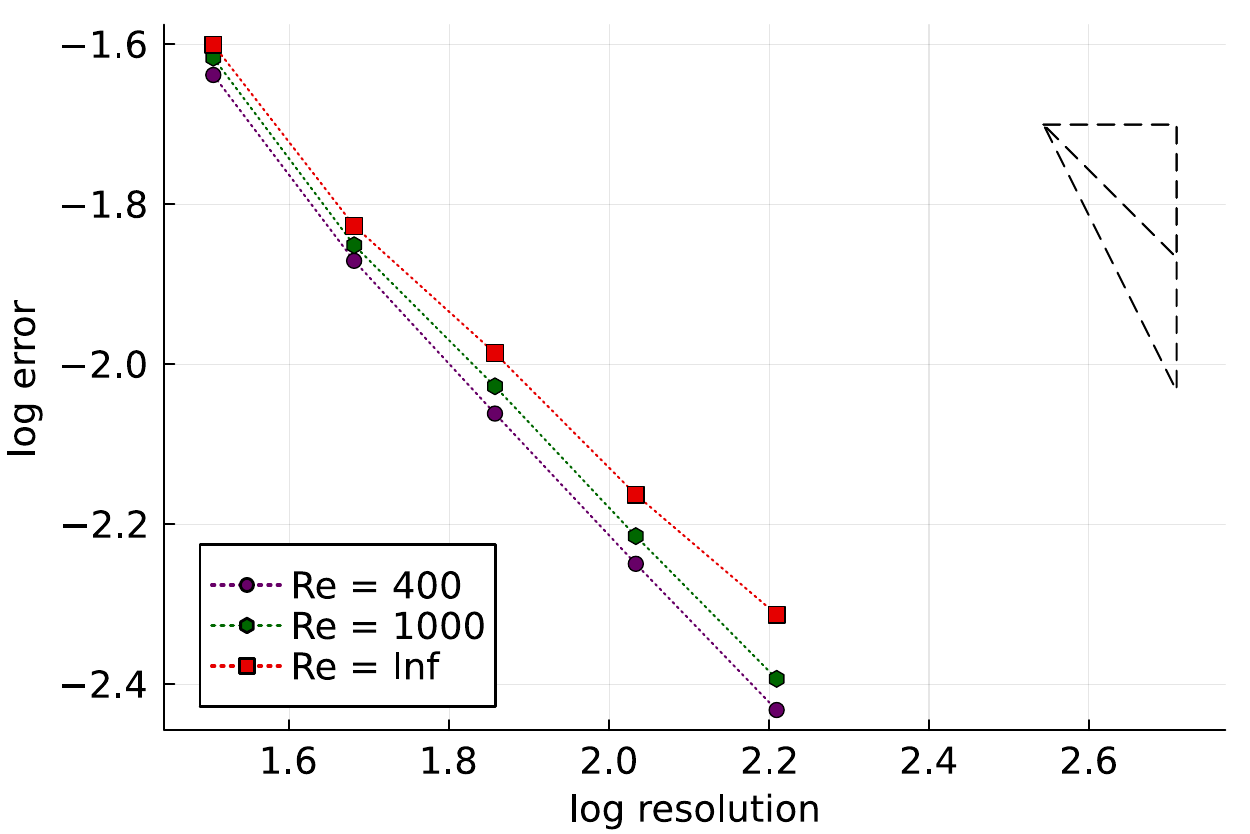}
			\subcaption{velocity error}
		\end{subfigure}%
		\begin{subfigure}{0.5\linewidth}
			\includegraphics[width=\linewidth]{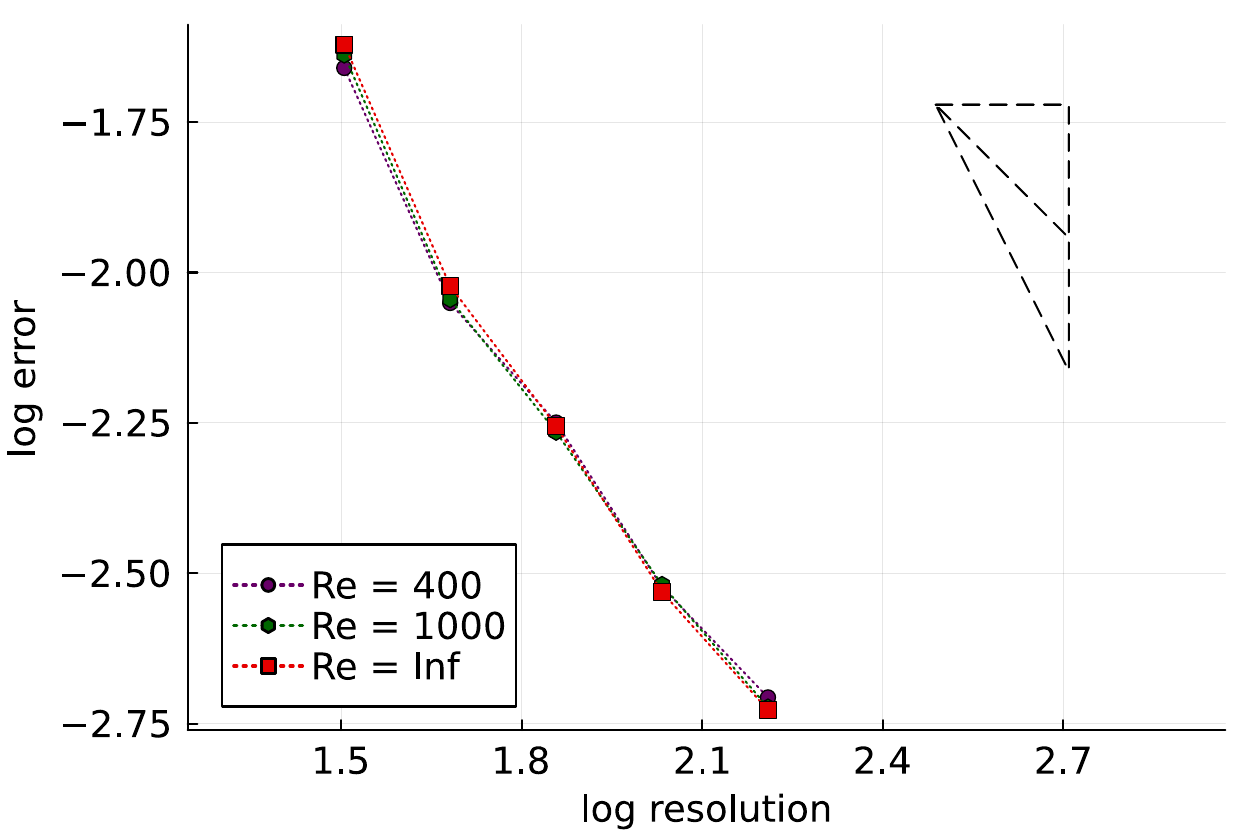}
			\subcaption{pressure error}
		\end{subfigure}
		\caption{The convergence rates for the Taylor-Green vortex benchmark.}
		\label{fig:tagr_convergence}
	\end{figure}
	
	\begin{table}[!htbp]
		\centering
		\begin{tabular}{|l|ll|ll|ll|}
			\hline
			& \multicolumn{2}{l|}{Re = 400}            & \multicolumn{2}{l|}{Re = 1000}           & \multicolumn{2}{l|}{Re = Inf}            \\ \hline \hline
			N   & \multicolumn{1}{l|}{velocity} & pressure & \multicolumn{1}{l|}{velocity} & pressure & \multicolumn{1}{l|}{velocity} & pressure \\ \hline \hline
			32  & \multicolumn{1}{l|}{2.30E-02} & 2.19E-02 & \multicolumn{1}{l|}{2.30E-02} & 2.30E-02 & \multicolumn{1}{l|}{2.51E-02} & 2.39E-02 \\ \hline
			48  & \multicolumn{1}{l|}{1.35E-02} & 8.90E-03 & \multicolumn{1}{l|}{1.35E-02} & 9.03E-03 & \multicolumn{1}{l|}{1.49E-02} & 9.49E-03 \\ \hline
			72  & \multicolumn{1}{l|}{8.67E-03} & 5.63E-03 & \multicolumn{1}{l|}{8.67E-03} & 5.43E-03 & \multicolumn{1}{l|}{1.03E-02} & 5.56E-03 \\ \hline
			108 & \multicolumn{1}{l|}{5.63E-03} & 3.00E-03 & \multicolumn{1}{l|}{5.63E-03} & 3.03E-03 & \multicolumn{1}{l|}{6.87E-03} & 2.94E-03 \\ \hline
			162 & \multicolumn{1}{l|}{3.69E-03} & 1.97E-03 & \multicolumn{1}{l|}{3.69E-03} & 1.89E-03 & \multicolumn{1}{l|}{4.86E-03} & 1.87E-03 \\ \hline \hline
			EOC & \multicolumn{1}{l|}{1.117}    & 1.456    & \multicolumn{1}{l|}{1.089}    & 1.502    & \multicolumn{1}{l|}{1.000}    & 1.545    \\ \hline
		\end{tabular}
		\caption{The convergence rate in Taylor-Green vortex test. Number $N$ determines the number of generating seeds per side of the domain. EOC is the experimental order of convergence.}
		\label{tab:tagr}
	\end{table}
	
	\subsection{Rayleigh-Taylor instability}
	Following \cite{rayleigh1882investigation}, we present a Rayleigh-Taylor instability between two immiscible fluids without surface tension. The computational domain $\Omega = (0,1)\times (0,2)$ is divided into an area occupied by a lighter and a heavier fluid, where the initial distribution of density is given by
	\begin{align}
		\rho(x,y) &= \begin{cases}
			1.8 & y > \phi(x)\\
			1 & y < \phi(x)
		\end{cases},\\
		\phi(x) &= 1 - 0.15 \cos (2\pi x). 
	\end{align}
	The fluid is viscous and subjected to a uniform gravitational field such that $\mathrm{Re} = 420$ and $\mathrm{Fr} = 1$. No-slip condition is imposed on boundaries. The incompressible constraint is reproduced by stiffened gas equation with $\mathrm{Ma} = 10^{-3}$. The result, depicted in Figure \ref{fig:rti}, is in qualitative agreement with the reference SPH simulation \cite{rayleigh1882investigation}, and it is also very much consistent with the result of the fully incompressible version of SILVA \cite{SILVAinc}.
	
	\begin{figure}[!htbp]
		\centering
		\begin{subfigure}{0.33\linewidth}
			\centering
			\includegraphics[width=0.95\linewidth]{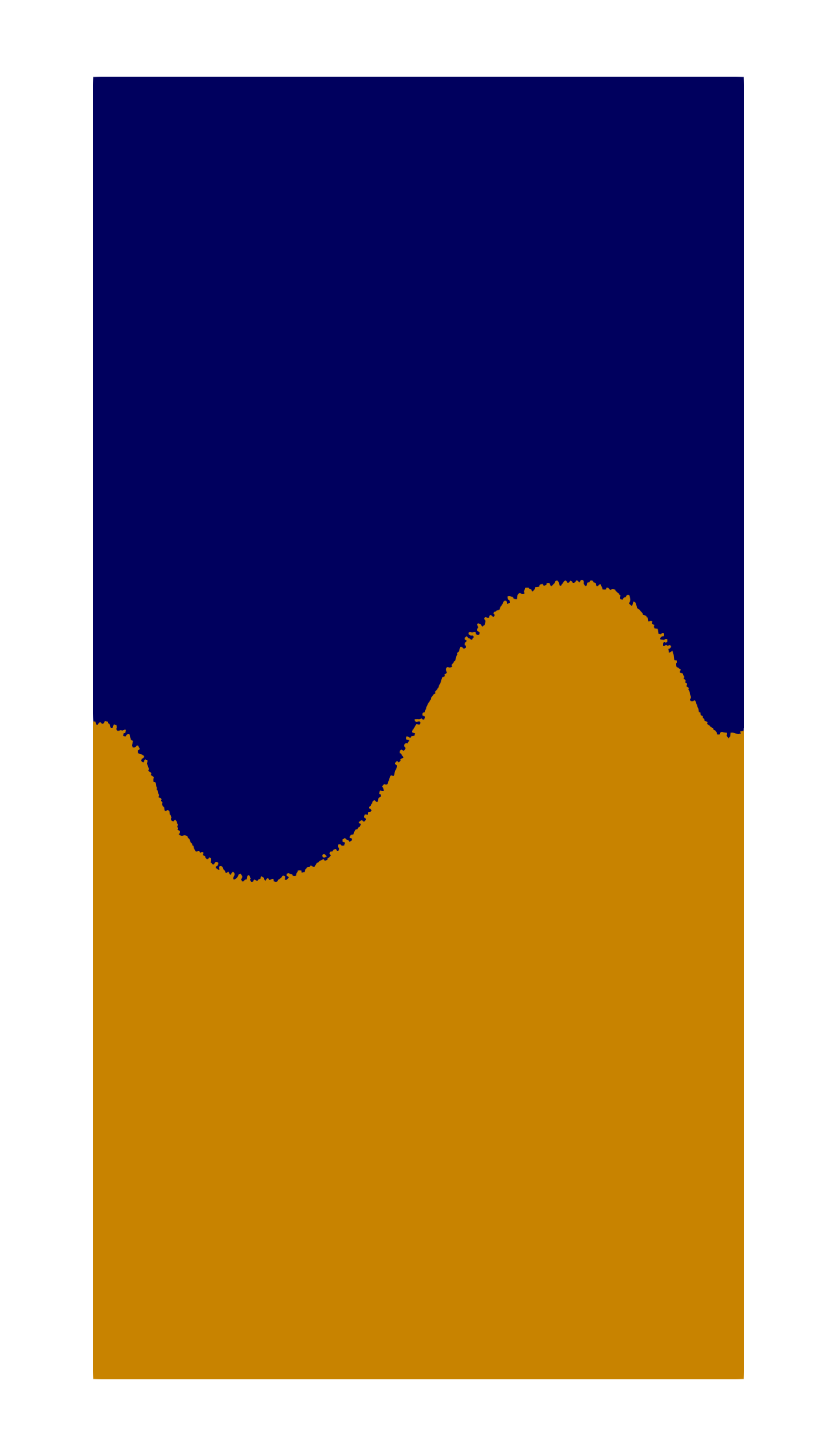}
			\subcaption{$t=1$}
		\end{subfigure}%
		\begin{subfigure}{0.33\linewidth}
			\centering
			\includegraphics[width=0.95\linewidth]{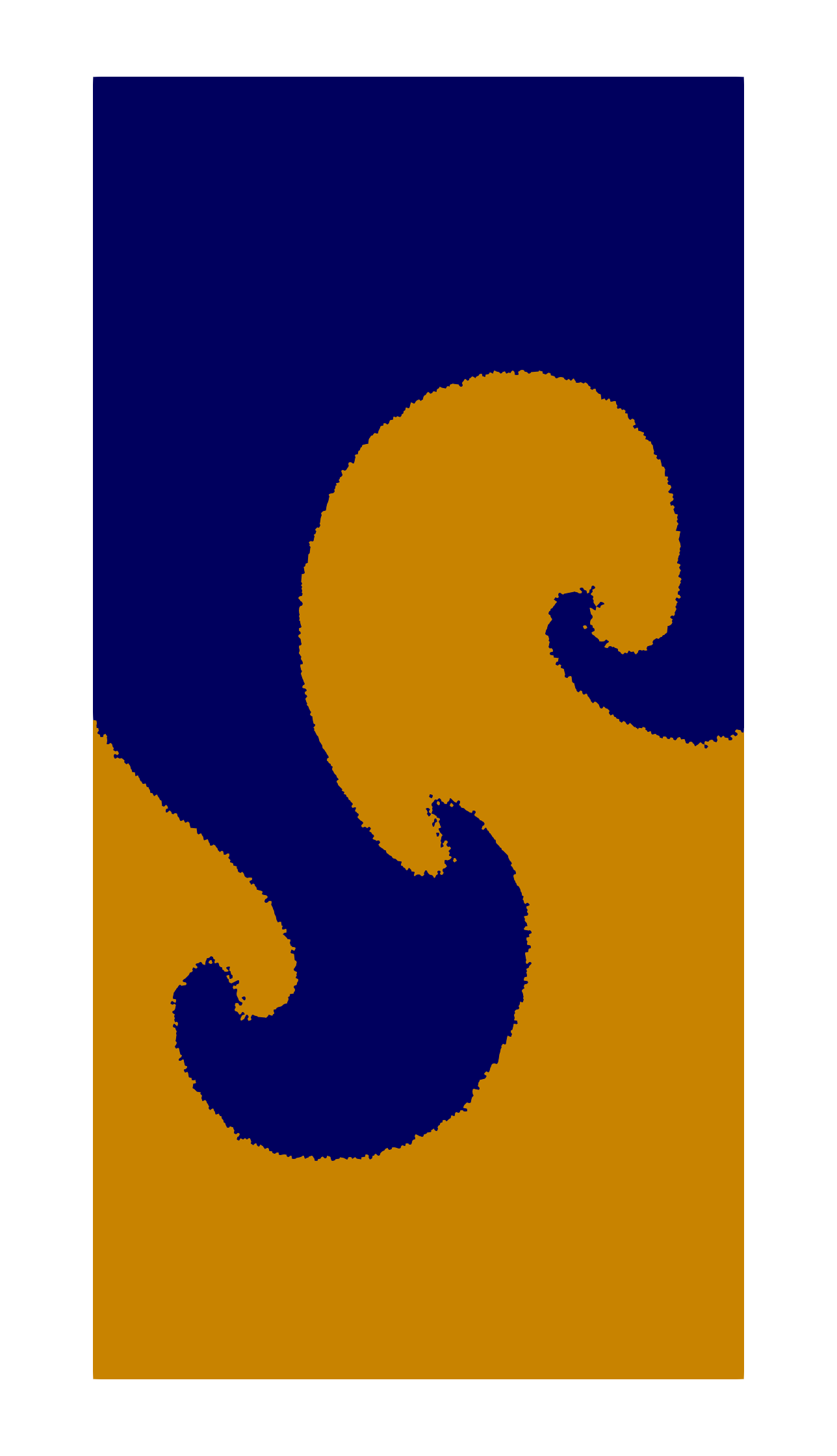}
			\subcaption{$t=3$}
		\end{subfigure}%
		\begin{subfigure}{0.33\linewidth}
			\centering
			\includegraphics[width=0.95\linewidth]{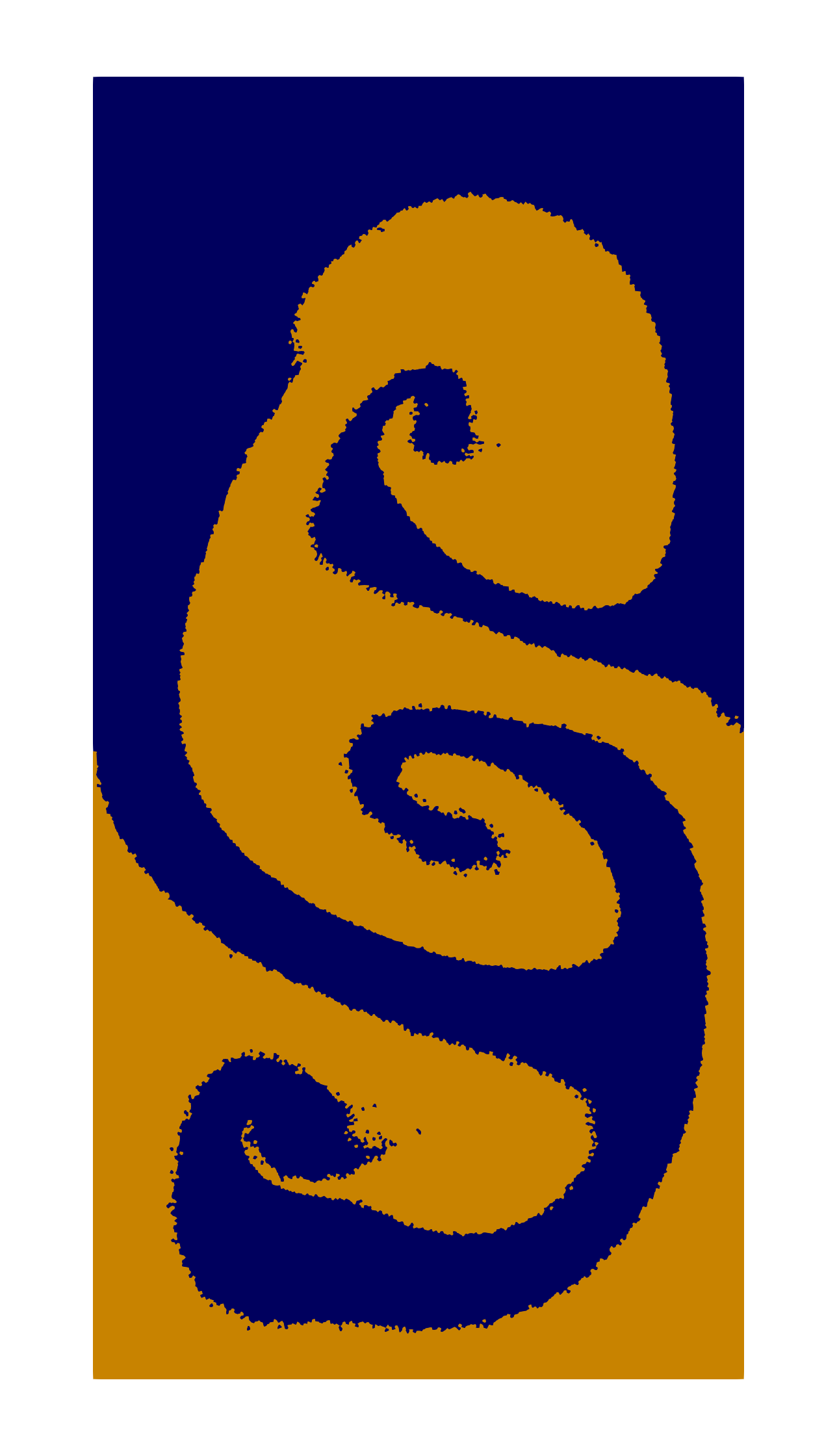}
			\subcaption{$t=5$}
		\end{subfigure}
		\caption{Snapshots of Rayleigh-Taylor instability. The lighter and heavier fluid are indicated by orange and dark blue color respectively. The mesh contains $80000$ Voronoi cells.}
		\label{fig:rti}
	\end{figure}
	
	\subsection{Lid-driven cavity}
	
	To assess the accuracy of the viscous force approximation in a quantitative benchmark, we refer to the standard lid-driven cavity test. The domain is a square $\Omega = (0,1)^2$ with no-slip boundary condition at the bottom and sides, and a Dirichlet condition
	\begin{equation}
		\vv_D = \begin{pmatrix}
			1\\
			0
		\end{pmatrix}
	\end{equation}
	at the top. Similarly to the Taylor-Green vortex \ref{sec:tagr}, the incompressible fluid is approximated via the stiffened gas model with $\mathrm{Ma} = 0.001$. A steady solution is sought, which can be obtained by our dynamic solver when we choose a sufficiently large termination time. The resolution was 200x200 cells. In the final state, we measure the $v_x$ and $v_y$ velocity along the vertical and the horizontal symmetry axis. In Figure \ref{fig:ldc_graph}, the results are in good agreement with the referential solution by Abdelmigid et al. \cite{abdelmigid2017revisiting}. Figure \ref{fig:ldc_plot} depicts the velocity magnitude and streamlines, indicating the presence of corner vortices.

	\begin{figure}[!htbp]
		\centering
		\begin{subfigure}{0.33\linewidth}
			\centering
			\includegraphics[width=\linewidth]{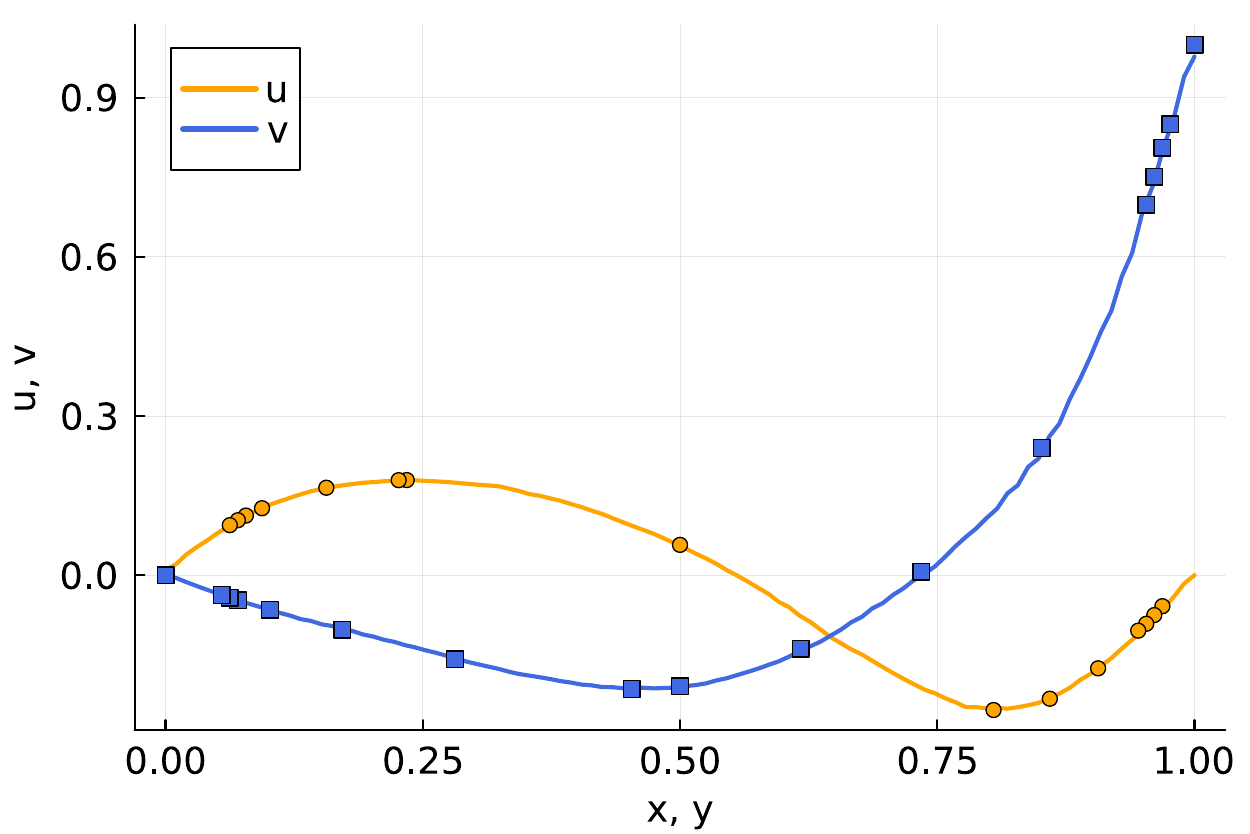}
			\subcaption{$\mathrm{Re}=100$}
		\end{subfigure}%
		\begin{subfigure}{0.33\linewidth}
			\centering
			\includegraphics[width=\linewidth]{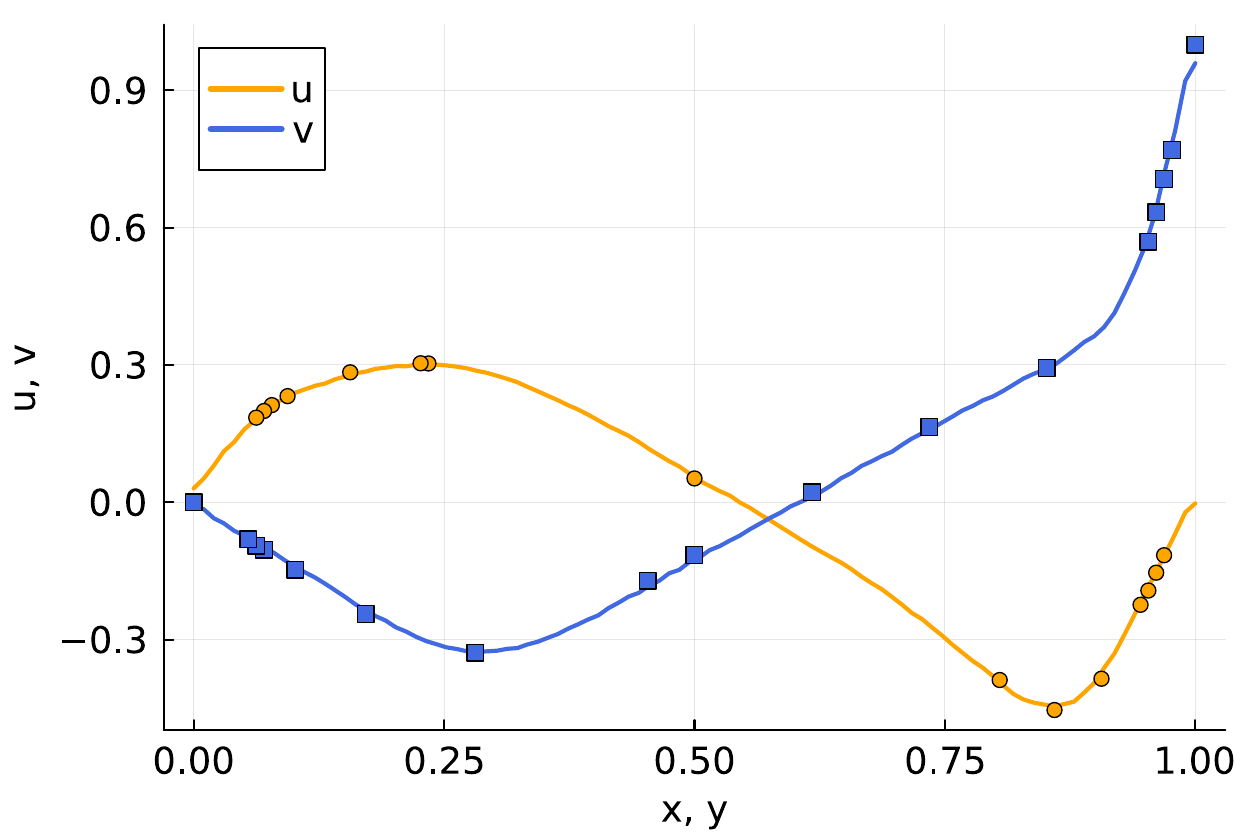}
			\subcaption{$\mathrm{Re}=400$}
		\end{subfigure}%
		\begin{subfigure}{0.33\linewidth}
			\centering
			\includegraphics[width=\linewidth]{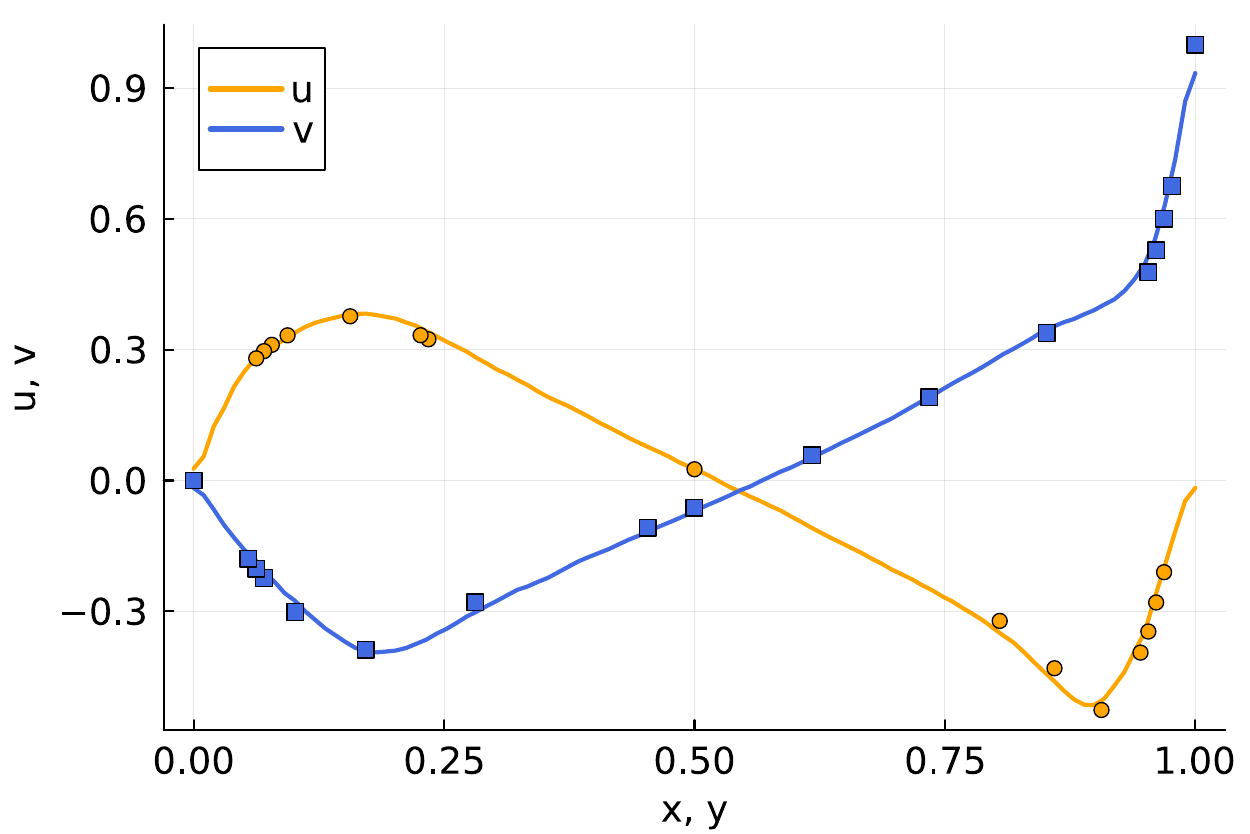}
			\subcaption{$\mathrm{Re}=1000$}
		\end{subfigure}
		\caption{Plots of $x$ and $y$ components of velocity in the lid-driven cavity benchmark along central axes. Lines show the numerical result using SILVA and marker indicate the reference solution. The terminating time for the simulation was selected as $t_\mathrm{end} = 0.1\mathrm{Re}$. The point values were obtained using moving least square interpolation.}
		\label{fig:ldc_graph}
	\end{figure}
	\begin{figure}[!htbp]
		\centering
		\includegraphics[width=0.4\linewidth]{images/inferno_colormap.png}
		\begin{subfigure}{0.33\linewidth}
			\centering
			\includegraphics[width=\linewidth]{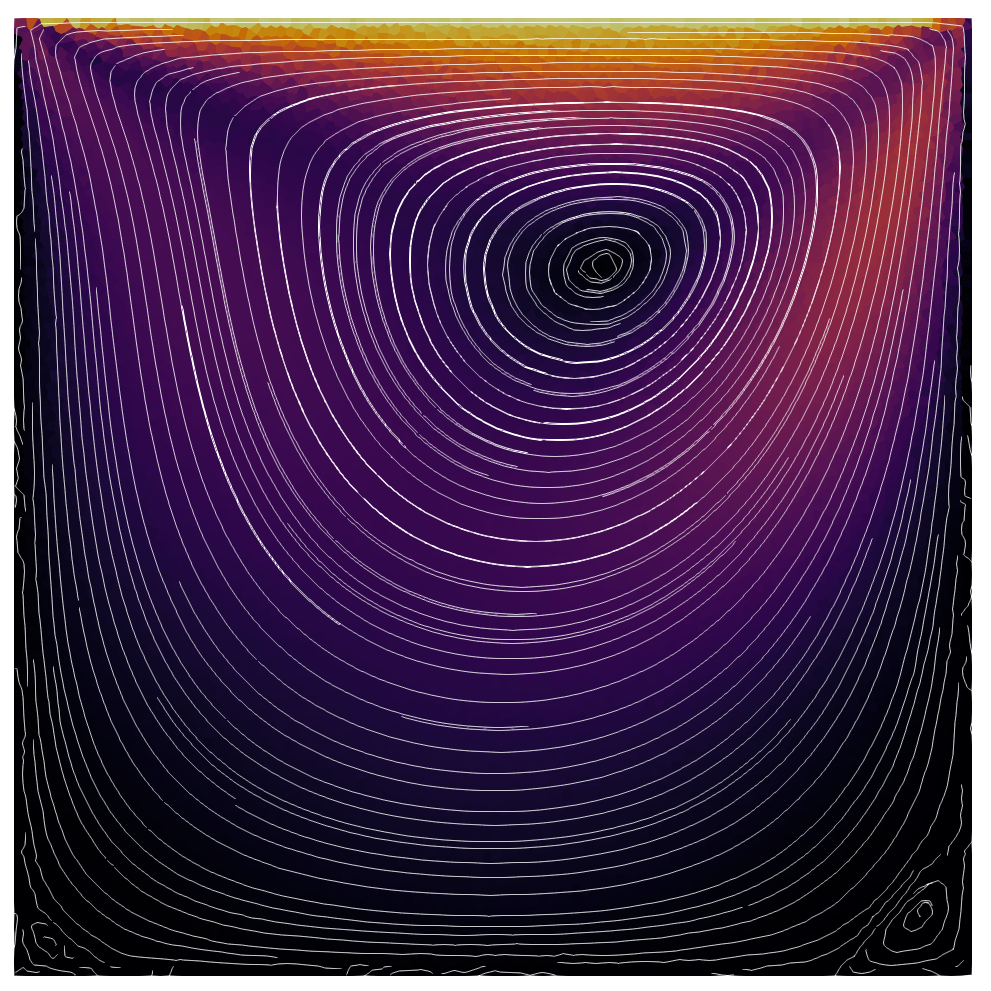}
			\subcaption{$\mathrm{Re}=100$}
		\end{subfigure}%
		\begin{subfigure}{0.33\linewidth}
			\centering
			\includegraphics[width=\linewidth]{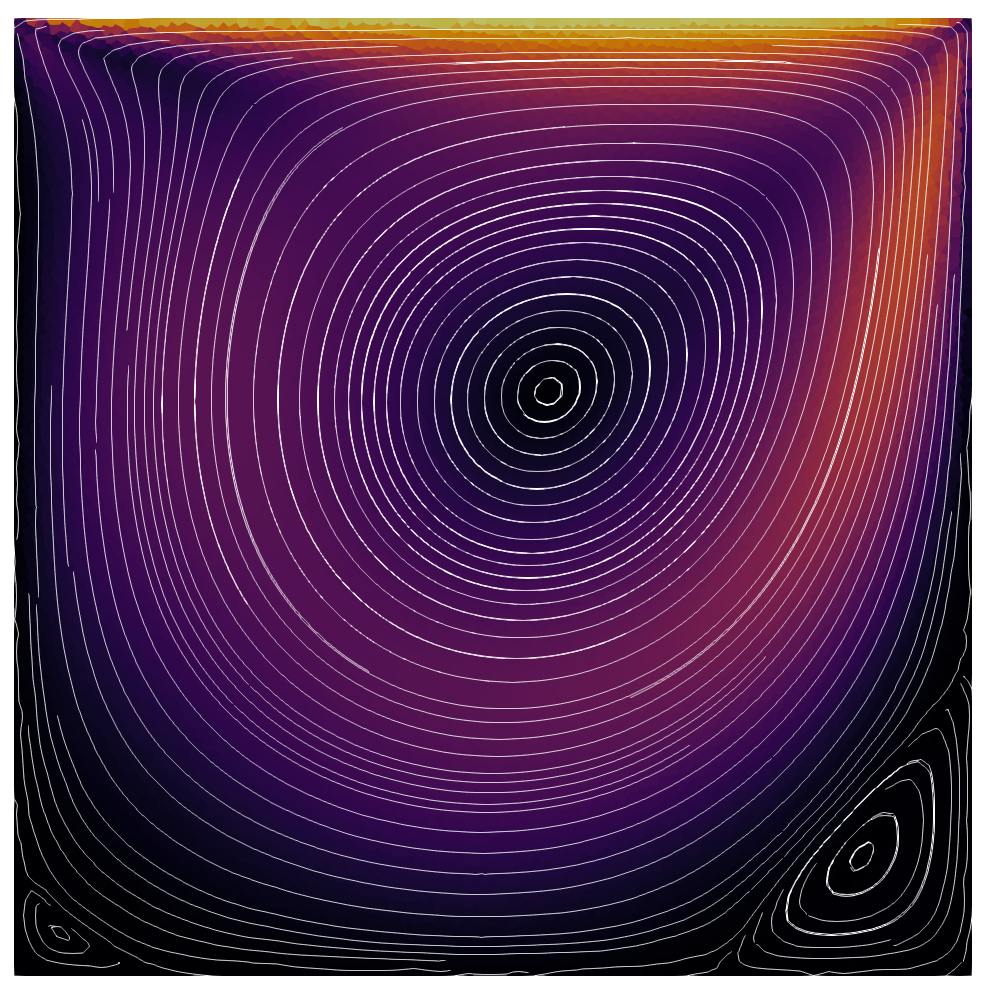}
			\subcaption{$\mathrm{Re}=400$}
		\end{subfigure}%
		\begin{subfigure}{0.33\linewidth}
			\centering
			\includegraphics[width=\linewidth]{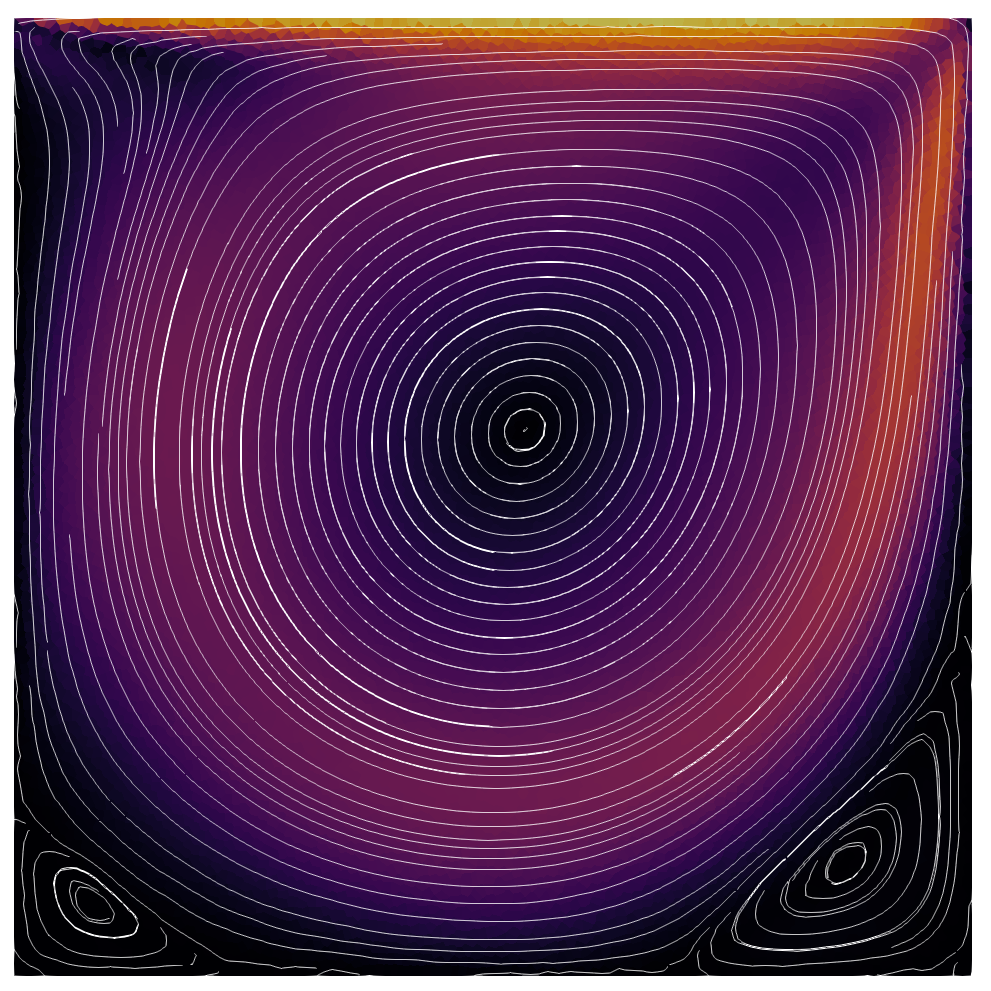}
			\subcaption{$\mathrm{Re}=1000$}
		\end{subfigure}
		\caption{The velocity magnitude and streamlines in lid-driven cavity.}
		\label{fig:ldc_plot}
	\end{figure}
	
	
	\section{Conclusion}
	This paper follows up on the recent investigations concerning numerical methods based on moving Voronoi meshes. The main novelties lie in the design of a semi-implicit time integrator for a fully compressible viscous fluid (Section \ref{sec:reversible}), and the multi-phase correction technique for the mesh relaxation step (Section \ref{sec:multiproj}). Seven different benchmarks were presented, which include shocks and multi-phase flows as well as a numerical convergence study. We think that the results are encouraging and we hope that this manuscript will elevate interest in Lagrangian Voronoi methods.
	
	So far, our numerical method is only applicable for planar flows but regarding the extension to three-dimensional flows, there is no foreseeable obstacle because the underlying mathematics is general and works in arbitrary dimension. The implementation of free surfaces with high density ratios remains an obstacle. Our preliminary experiments indicate problems with stability of the interface, so adding some (artificial) surface tension might be necessary. Further research is needed also to obtain a second order method in space and time. Moreover, our SILVA algorithm relies on an explicit discretization of the viscous terms. An implicit treatment could lead to a substantially larger, and thus expensive, time step (because the velocity field has three times more degrees of freedom than the pressure) but it is needed to simulate high-viscosity fluids and to improve the computational efficiency of the overall method.
	
	\section{Acknowledgments}
    This work received financial support by the Italian Ministry of University 
	and Research (MUR) in the framework of the PRIN 2022 project No. 2022N9BM3N. WB acknowledges research funding by Fondazione Cariplo and Fondazione CDP (Italy) under the project No. 2022-1895 and by the "Agence Nationale de la Recherche" (ANR) project No. ANR-23-EXMA-0004. IP and WB are members of the INdAM GNCS group in Italy.
	
	\noindent \section*{In memoriam}
	
	\noindent This paper is dedicated to the memory of Prof. Arturo Hidalgo L\'opez
	($^*$July 03\textsuperscript{rd} 1966 - $\dagger$August 26\textsuperscript{th} 2024) of the Universidad Politecnica de Madrid, 
	organizer of HONOM 2019 and active participant in many other editions of HONOM.
	Our thoughts and wishes go to his wife Lourdes and his sister Mar\'ia Jes\'us, whom he left behind.

\end{document}